\documentclass[12pt,oneside,reqno]{amsart}

\usepackage[utf8]{inputenc}
\usepackage[T1]{fontenc}
\usepackage{hyperref}
\usepackage{amsmath}
\usepackage{amsthm}
\usepackage{amsfonts}
\usepackage{amssymb}
\usepackage{amsbsy}
\usepackage{subcaption}
\usepackage{pdfpages}
\usepackage{fancyhdr}
\usepackage[hmargin=22mm,vmargin=22mm,bmargin=30mm,tmargin=25mm]{geometry}
\usepackage{geometry}
\usepackage{setspace}
\usepackage{xcolor}
\usepackage{tikz}
\usepackage{pifont}
\usepackage{graphicx}
\usepackage{tabularx}
\usepackage{epsfig}
\usepackage{nccmath}

\usepackage[normalem]{ulem}

\usetikzlibrary{matrix}

  \usepackage{appendix}

  \usepackage{enumerate}

  \usepackage[sort, numbers]{natbib}
  \usepackage{bibentry}

  \newtheorem{theorem}{Theorem}[section]
  \newtheorem*{theorem*}{Theorem}
  \newtheorem*{lemma*}{Lemma}
  \newtheorem{theoremm}[theorem]{Theorem}
  \newtheorem{lemma}[theorem]{Lemma}
  \newtheorem{proposition}[theorem]{Proposition}
  
  \newtheorem{corollary}[theorem]{Corollary}
  
  \theoremstyle{definition}
  \newtheorem{definition}[theorem]{Definition}

  \newtheorem*{remark}{Remark}

  \numberwithin{equation}{section}

  \newcommand{\N}{{\mathbb N}}
  \newcommand{\Z}{{\mathbb Z}}

  \newcommand{\R}{{\mathbb R}}
  \newcommand{\C}{{\mathbb C}}
  \newcommand{\D}{{\mathbb D}}
  \newcommand{\T}{{\mathbb T}}

  \newcommand{\E}{{\mathsf E}}
  \newcommand{\sfA}{\mathsf{A}}

    \newcommand{\overbar}[1]{\mkern 1.5mu\overline{\mkern-1.5mu#1\mkern-1.5mu}\mkern 1.5mu}

  \newcommand{\cP}{{\mathcal{P}}}

  \newcommand{\cW}{{\mathcal{W}}}

\usetikzlibrary{decorations.pathreplacing,arrows.meta}

	\newcommand{\RS}{\overline{\C}}

\definecolor{apricot}{rgb}{0.98,0.81,0.69}

  \renewcommand{\Re}{\operatorname{Re}}
  \renewcommand{\Im}{\operatorname{Im}}
  \renewcommand{\Cap}{\operatorname{Cap}}

\usepackage{soul}
\usepackage{ulem}

\usepackage{caption}
\DeclareCaptionLabelFormat{cont}{#1~#2\alph{ContinuedFloat}}
  \newcommand{\bbD}{\mathbb{D}}

\usetikzlibrary{shapes}
\usetikzlibrary{plotmarks}
  \allowdisplaybreaks

\author{Jacob S. Christiansen$^{1,5}$, Benjamin Eichinger$^{2,6}$, Olof Rubin$^{3,7}$, and Maxim Zinchenko$^{4,8}$}

\thanks{$^1$ Centre for Mathematical Sciences, Lund University, Box 118, 22100 Lund, Sweden.
E-mail: jacob\_stordal.christiansen@math.lth.se}
\thanks{$^2$ School of Mathematical Sciences, Lancaster University, Lancaster, LA1 4YF, United Kingdom
E-mail:  b.eichinger@lancaster.ac.uk}
\thanks{$^3$ Division of Probability, Mathematical Physics and Statistics, KTH Royal Institute of Technology, Lindstedtsvägen 25, 11428 Stockholm, Sweden.
E-mail: orubin@kth.se}
\thanks{$^4$ Department of Mathematics and Statistics, University of New Mexico, Albuquerque, NM 87131, USA; e-mail: maxim@math.unm.edu}
\thanks{$^5$ Research supported by VR grant 2023-04054 from the Swedish Research Council and in part by DFF research project 1026-00267B from the Independent Research Fund Denmark.}
\thanks{$^6$ Research supported by the Austrian Science Fund FWF, Project No. P33885}
\thanks{$^7$ Research supported by Odysseus Grant G0DDD23N from 
the Research Foundation Flanders (FWO) and The G\"{o}ran Gustafsson Foundation for Research in Natural Sciences and Medicine}
\thanks{$^8$ Research supported in part by the Simons Foundation Grant MP-TSM-00002651.}

\title{Weighted residual polynomials on a circular arc}

\begin{document}

\begin{abstract}
We study the behavior of weighted residual polynomials on circular arcs, including weighted Chebyshev polynomials. For weights given by reciprocals of polynomials, we establish Szeg\H{o}–Widom asymptotics. Extending our analysis to less regular weights, we determine the asymptotic behavior of the corresponding weighted Widom factors, generalizing results by Eichinger and Thiran et al. As an application, we derive the asymptotics of Widom factors on certain lemniscatic arcs.
\end{abstract}
	\medskip	
	
	
\medskip	
	

\maketitle

\section{Introduction}
Given a set $\E\subset \overline{\C}$, where $\overline{\C}:=\C\cup\{\infty\}$ denotes the Riemann sphere, we define $\|\cdot\|_{\E}$ as the supremum norm, meaning that for any function $f:\E\rightarrow \C$,
\[
   \|f\|_{\E} := \sup_{z\in \E}|f(z)|.
\]
The central object in this work is $\cP_n$, the set of algebraic polynomials of degree at most $n\in \N$. 
For every pair consisting of a polynomially convex compact set $\E\subset \C$ (i.e., a compact set such that its complement $\C \setminus \E$ is connected) and a fixed bounded weight function $w: \E \to [0, \infty)$, we can associate a sequence of \emph{weighted residual polynomials}. 
For $n\in \N$ and $z_0\in \RS\setminus \E$, we will use the notation $R_{n}^{\E,w}(\cdot,z_0)$ to denote the unique polynomial of degree at most $n$ satisfying
\begin{align*} 
   (i)& \quad \|wR_{n}^{\E,w}(\cdot ,z_0)\|_{\E}=\sup_{z\in \E}\left|w(z)R_n^{\E,w}(z,z_0)\right|\leq 1,  
   \mbox{ and }  \\
   (ii)& \quad R_{n}^{\E,w}(z_0,z_0) = \sup\bigl\{|P(z_0)|: P\in \cP_n,\, \|wP\|_{\E}\leq 1\bigr\}. 
\end{align*}
For a polynomial $P$, the value $P(\infty)$ is to be understood as its leading coefficient. The normalization is specified so that the residual polynomial is positive at the point $z_0$. For a proof of existence and uniqueness, we refer the reader to \cite[Theorem 2.1\,b)]{christiansen-simon-zinchenko-V}. 

Alternatively, the characterization in $(i)$--$(ii)$ can be reformulated in its dual form: determine the polynomial that minimizes $\|wP\|_{\E}$ among all polynomials $P$ satisfying $P(z_0) = 1$.
In fact, this polynomial is given by
\begin{equation}
	T_n^{\E,w}(z,z_0)=\frac{R_n^{\E,w}(z,z_0)}{R_n^{\E,w}(z_0,z_0)},
	\label{eq:chebyshev_residual_relation}
\end{equation}
and we note that
\begin{equation}
	\|wT_n^{\E,w}(\cdot,z_0)\|_{\E} = \frac{1}{R_n^{\E,w}(z_0,z_0)}.
	\label{eq:dual_relation}
\end{equation}
The polynomials $T_n^{\E,w}:=T_{n}^{\E,w}(\cdot,\infty)$, corresponding to $z_0=\infty$, have received considerable attention in the literature (see, e.g., \cite{lorentz86,smirnov-lebedev68,achieser56,christiansen-simon-zinchenko-I,novello-schiefermayr-zinchenko21}). These so-called \emph{weighted Chebyshev polynomials} are fundamental objects in approximation theory. The more general residual polynomials also play a significant role in computational science, particularly in Krylov subspace iterations (see, e.g., \cite{driscoll-toh-trefethen98, kuijlaars06, fischer96}). 

Classical considerations include determining residual polynomials relative to the interval \([-1,1]\). For points \(x_0 \in (-\infty, -1) \cup (1, \infty)\), the alternating property of residual polynomials relative to real sets imply that  
\[
R_n^{[-1,1]}(x, x_0) = 2^{n-1}T_n(x),
\]
where \(T_n(x) := T_n^{[-1,1]}(x)\) denotes the classical Chebyshev polynomials on $[-1,1]$ (see \cite{markoff16, christiansen-simon-zinchenko-V}). The properties of residual polynomials relative to \([-1,1]\) for non-real points are more intricate. Specifically, purely imaginary points were studied in \cite{freund88, freund-ruscheweyh}, while the case of arbitrary points was addressed by Yuditskii in \cite[Proposition 1]{yuditskii99}. The latter reference will be particularly relevant for the discussion that follows.

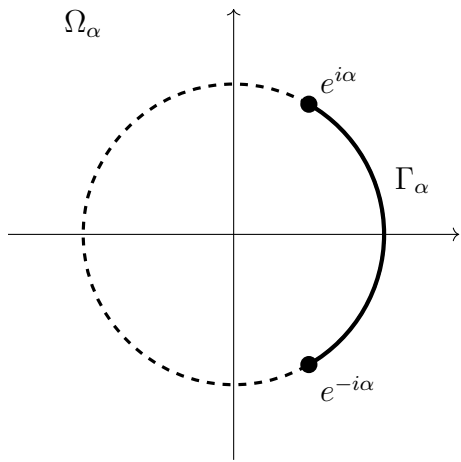
\begin{figure}[h!]
	\centering
	\begin{tikzpicture}[scale=2]
\begin{scope}[xshift=5cm]
    \draw[very thick, dashed] (1,0) arc[start angle=0, end angle=360, radius=1];
    \draw[ultra thick] (1,0) arc[start angle=0, end angle=60, radius=1];
    \draw[ultra thick] (1,0) arc[start angle=0, end angle=-60, radius=1];

    \draw[->] (-1.5,0) -- (1.5,0) node[right]{};
    \draw[->] (0,-1.5) -- (0,1.5) node[above]{} ;
	\draw[thick,fill=black] (0.5,0.866) circle(0.05)node[above right] {$e^{i\alpha}$};
    \draw[thick,fill=black] (0.5,-0.866) circle(0.05)node[below right] {$e^{-i\alpha}$};

    \node[above right] at (1,0.2) {$\Gamma_\alpha$};
	
    \node at (-1,1.4) {$\Omega_\alpha$};
\end{scope}
\end{tikzpicture}

\caption{A circular arc $\Gamma_\alpha = \{e^{it}: -\alpha\leq t \leq \alpha\}$.}
\label{fig:circular_arcs}
\end{figure}

Our interest is directed toward weighted residual polynomials on connected subsets of the unit circle $\T$. Such studies extend the results of \cite{eichinger17}. Any number $0<\alpha<\pi$ gives rise to a circular arc as pictured in Figure \ref{fig:circular_arcs}; it is determined by the formula
 \begin{equation}
	\Gamma_\alpha := \{e^{it}: -\alpha\leq t \leq \alpha\}.
\end{equation}
Given a weight function $w:\Gamma_\alpha\rightarrow [0,\infty)$, we wish to describe the asymptotic behavior of the corresponding weighted residual polynomials. 

To explain the novel aspects of this work, we employ concepts from potential theory. For background on this topic, we refer the reader to \cite{ransford95, garnett-marschall05, landkof72, armitage-gardiner01}. 
Let $\Omega$ be a proper subdomain of $\overline{\C}$ and let $\zeta\in\Omega$. Suppose that $\partial\Omega$ is non-polar, meaning it has positive logarithmic capacity. Then $\Omega$ admits a unique Green's function with pole at $\zeta$, which we denote by $g_{\Omega}(\cdot, \zeta)$. 
The harmonic measure of a Borel set $I\subset\partial\Omega$ with respect to the base point $\zeta$ is denoted by $\omega(\zeta, I; \Omega)$. Recall that if $f:\partial\Omega\rightarrow \R$ is continuous, then the function $h: \Omega\to \R$ defined by
\begin{equation}
	h(z) := \int_{\partial \Omega} f(x)\,\omega(z, dx; \Omega)
	\label{eq:harmonic_measure_property}
\end{equation}
is harmonic in $\Omega$ and assumes the boundary value $f(x)$ at all regular points of $\partial\Omega$.
In the specific case where $\Omega := \overline{\C} \setminus \E$ is the complement of a compact, non-polar, and polynomially convex set $\E \subset \C$, it follows that the boundary of $\Omega$ coincides with that of $\E$. This allows the Green's function for $\Omega$ with pole at infinity to be expressed directly via the potential theory of $\E$:
\begin{equation}
g_{\Omega}(z, \infty) = \int \log |z - x|\,d\mu_\E(x) - \log \Cap(\E), \quad z \in \Omega,
\end{equation}
where $\mu_\E$ is the equilibrium measure of $\E$ and $\Cap(\E)$ is its logarithmic capacity. The equilibrium measure, in turn, coincides with the harmonic measure at $\zeta=\infty$ for the domain $\Omega$.





Introducing the notation 
\begin{equation}
   \Omega_\alpha := \RS\setminus \Gamma_\alpha
   \; \mbox{ and } \; 
   \C_+ := \{z:\Re z>0\},
\end{equation}    
we will associate the variables $u$ and $\lambda$ with these domains, respectively. 
The first study of residual polynomials relative to $\Gamma_\alpha$ was conducted in \cite{eichinger17}. Recall that the reproducing kernel for the weighted Bergman space on $\C_+$ with respect to the measure $\lambda\,  dA(\lambda)$, where $dA$ denotes the area measure, is given by
\begin{equation}
   k_{\C_+}(\lambda,\lambda_0) := \frac{2\lambda\overline{\lambda}_0}{(\lambda+\overline{\lambda}_0)^2}.
\end{equation}
For $u\in \Omega_\alpha$, it is proven in \cite[Theorem 1.3]{eichinger17} that  
\begin{equation}
	\|T_n^{\Gamma_\alpha}(\cdot,u)\|_{\Gamma_\alpha} = \frac{1}{R_n^{\Gamma_\alpha}(u,u)} \sim e^{-ng_{\Omega_\alpha}(u,\infty)}k_{\C_+}\bigl(\lambda(u),\lambda(u)\bigr)^{-1},
	\label{eq:eichinger_asymptotics}
\end{equation}
where $\lambda:\Omega_\alpha\rightarrow \mathsf{S}_{\pi/4}:=\{z\in \C_+: |\arg z|< \pi/4\}$ is the conformal map given by
\begin{equation}
\label{la}
   \lambda(u) := \left(\frac{ue^{i\alpha}-1}{u-e^{i\alpha}}\right)^{1/4}.
\end{equation}  
As usual, the symbol $\sim$ denotes that the ratio of the two expressions tends to $1$ as $n\rightarrow \infty$. 

\begin{center}
\begin{tikzpicture}

\begin{scope}
	\draw[dashed, ultra thick] (2,0) arc[start angle=0, end angle=45, radius=2];
    \draw[dashed, ultra thick] (2,0) arc[start angle=0, end angle=-45, radius=2];

    \draw[thick, ->] (-3,0) -- (3,0) node[right] {$\text{Re}(u)$};
    \draw[thick, ->] (0,-3) -- (0,3) node[above] {$\text{Im}(u)$};
    
	\draw[thick,fill=black] (1.4142135624,1.4142135624) circle(0.05)node[above right] {$e^{i\alpha}$};
    \draw[thick,fill=black] (1.4142135624,-1.4142135624) circle(0.05)node[below right] {$e^{-i\alpha}$};

    \node[below right] at (2,0) {$1$};
	
    \node at (-1.4,1.7) {$\Omega_\alpha$};
\end{scope}

\begin{scope}[xshift=8cm]
    \draw[dashed, ultra thick] (0,0) -- (3,3);
    \draw[dashed, ultra thick] (0,0) -- (3,-3);
    
    \fill[gray!20] (0,0) -- (3,3) -- (3,-3) -- cycle;
  	\draw[thick,->] (-3,0) -- (3,0) node[below right] {$\Re(\lambda)$};
  	\draw[thick,->] (0,-3) -- (0,3) node[above left] {$\Im(\lambda)$};

  	\node[left,black] at (3,1.3) {$\mathsf{S}_{\pi/4}$};
  	\draw[thick,fill=black] (0,0) circle(0.05)node[below left] {$\lambda(e^{-i\alpha})$};
\end{scope}

\draw[->, thick] (3,1) to[out=30, in=150] (7,1);

\node[above] at (5,1.9) {$\lambda(u)$};
\end{tikzpicture}	
\end{center}


First, we extend this result to weighted polynomials, where the weight is the multiplicative inverse of a polynomial. Ultimately, this will enable us to prove that for $u\in \Omega_\alpha$ and a general class of weights, whose properties are detailed in Theorem \ref{thm:main_norm_asymptotics},
\begin{equation}
	\|wT_n^{\Gamma_{\alpha},w}(\cdot ,u)\|_{\Gamma_\alpha} \sim e^{-ng_{\Omega_\alpha}(u,\infty)}k_{\C_+}\bigl(\lambda(u),\lambda(u)\bigr)^{-1}\exp\left(\int_{\Gamma_\alpha}\log w(x)\,\omega(u,dx,\Omega_{\alpha})\right).
	\label{eq:residual_norm_asymptotics}
\end{equation}
In the special case $u = \infty$, we obtain the asymptotic formula
\begin{equation}
	\lim_{n\rightarrow \infty}\frac{\|wT_n^{\Gamma_{\alpha},w}\|_{\Gamma_\alpha}}{\Cap(\Gamma_\alpha)^n}= 2\cos^2(\alpha/4)\exp\left(\int_{\Gamma_\alpha}\log w(x)\,d\mu_{\Gamma_\alpha}(x)\right).
	\label{eq:chebyshev_norm_asymptotics}
\end{equation}
This should be compared with the unweighted ($w\equiv 1$) asymptotic formula
\begin{equation}
\lim_{n\rightarrow \infty}\frac{\|T_n^{\Gamma_\alpha}\|_{\Gamma_\alpha}}{\Cap(\Gamma_\alpha)^n} = 2\cos^2(\alpha/4)
\end{equation}
that was derived in \cite{thiran-detaille91}. 
A deeper understanding of the monotonicity properties of this convergence was later obtained through \cite[Theorem 8]{schiefermayr-zinchenko21}. The context for this result is historically significant. For a bounded upper semicontinuous function $w:\Gamma\rightarrow [0,\infty)$ with $\log w \in L^1(\mu_\Gamma)$, Widom established the general upper bound \cite[Theorem 11.4]{widom69}:
\begin{equation}
\limsup_{n \to \infty} \frac{\|wT_n^{\Gamma,w}\|_{\Gamma}}{\Cap(\Gamma)^n} \leq 2\exp\left(\int_{\Gamma} \log w(x) \, d\mu_{\Gamma}(x)\right).
\label{eq:alpan_asymptotics}
\end{equation}
In the unweighted case, the integral vanishes, reducing this bound to 2. Motivated by this and the known result for an interval, Widom conjectured that for any sufficiently smooth Jordan arc $\Gamma$, the limit would always be exactly 2, i.e., that $\lim_{n\rightarrow \infty}\|T_n^{\Gamma}\|_{\Gamma}/\Cap(\Gamma)^n = 2$. The example under discussion is notable because it was the first to disprove this conjecture.

In sharp contrast to Widom's conjecture, a recent result by Alpan \cite[Theorem 1.3]{alpan22} proves that the inequality in \eqref{eq:alpan_asymptotics} is strict if there is an interior point $z\in \Gamma$ (i.e., a point other than its endpoints) for which
\begin{equation}
\frac{\partial g_{\Omega_\Gamma}}{\partial n_+}(z,\infty) \neq \frac{\partial g_{\Omega_\Gamma}}{\partial n_-}(z,\infty), \quad \Omega_\Gamma = \RS\setminus \Gamma.
\label{eq:not-s-property-arc}
\end{equation}
Here, the normal derivatives are taken with respect to the two different sides of the arc. As later observed in \cite[Theorem 1.1]{christiansen-eichinger-rubin23}, the only arc for which equality in \eqref{eq:not-s-property-arc} holds at all interior points is a straight line segment, which therefore constitutes the sole arc where Widom's conjecture is valid.


When it comes to lower bounds, we may also consider 
any weight function \( w: \E \to [0,\infty) \) for which \( \log w \in L^1(\mu_\E) \). In this case, Szeg\H{o}'s inequality holds:
\begin{equation}
\|wT_n^{\E,w}\|_{\E} \geq \Cap(\E)^n \exp\left(\int_{\E} \log w(x) \, d\mu_\E(x)\right),
\label{eq:szego_inequality}
\end{equation}
as shown in \cite[Theorem 13]{novello-schiefermayr-zinchenko21}. Based on \eqref{eq:alpan_asymptotics} and \eqref{eq:szego_inequality}, it is natural to introduce the \emph{weighted Widom factors} as
\begin{equation}
\mathcal{W}_n(\E,w) := \frac{\|wT_n^{\E,w}\|_{\E}}{\Cap(\E)^n}.
\label{eq:widom_factor}
\end{equation}
The unweighted Widom factors, corresponding to $w\equiv 1$, will be denoted by $\cW_n(\E)$. For an arc \( \Gamma \), combining \eqref{eq:alpan_asymptotics} and \eqref{eq:szego_inequality} implies that all limit points of \( \mathcal{W}_n(\Gamma,w) \) lie within the range
\begin{equation}
	\left[\exp\left(\int_{\Gamma} \log w(x) \, d\mu_{\Gamma}(x)\right), 2\exp\left(\int_{\Gamma} \log w(x) \, d\mu_{\Gamma}(x)\right)\right].
	\label{eq:upper_lower_bounds_widom_factors}
\end{equation}
We emphasize that exact formulas for limit points of the Widom factors 
associated with arcs are scarce in the literature. Notably, the case of an interval remains the only thoroughly studied example, largely due to the seminal work of Bernstein \cite{bernstein30, bernstein31}. For the norms of unweighted Chebyshev polynomials relative to an arc, the only case understood beyond an interval is that of a circular arc.

Exploring the implications of the results presented here, as well as those in \cite{eichinger17}, for analyzing convergence rates of Krylov subspace methods would be interesting but lies beyond the scope of this article. In particular, it would be of interest to investigate algorithms related to unitary matrices with restricted spectra. 
However, due to the numerical nature of such considerations, we have opted not to include them in the present discussion.

\subsection{Outline}
In Section \ref{sec:chebyshev_polynomials_system_of_arcs}, we analyze the asymptotic behavior of the norms of Chebyshev polynomials associated with subsets of $\{z:|z^m+1| = r\}$ for $r>0$. 
Our complete asymptotic analysis is collected in Theorem \ref{thm:Widom_factors_lemniscatic_arcs} which relies on the weighted formulas proven in Section \ref{sec:general_weight}. One of our goals is to analyze the effects of \eqref{eq:not-s-property-arc} on the potential limit points of the Widom factors for a connected system of arcs.


In Section \ref{sec:reciprocal_weight}, we will adopt the approach from \cite{eichinger17} where Szeg\H{o}--Widom asymptotics is established for circular arcs in the general residual polynomial setting. In fact, we will provide similar results for weight functions which are given as reciprocals of polynomials. To be more precise, in Theorem 
\ref{thm:weighted_residual_asymptotics} we determine the asymptotic behavior of 
   $R_n^{\Gamma_\alpha,1/P}(\cdot, u)$
when $P$ is a polynomial which does not vanish on the unit circle. 
This result enables us to determine the asymptotic behavior of 
   $\|T_n^{\Gamma_\alpha,1/P}(\cdot,u)/P\|_{\Gamma_\alpha}$
in Theorem \ref{thm:weighted_residual_asymptotics_at_point}, which serves as the foundation for the work presented in Section \ref{sec:general_weight}. There, we determine the asymptotic behavior of \(\|wT_n^{\Gamma_\alpha,w}(\cdot,u)\|_{\Gamma_\alpha}\), with a particular focus on the case \(u=\infty\), for a general class of weight functions, including those that may vanish on $\Gamma_\alpha$. Specifically, Theorem \ref{thm:main_norm_asymptotics} generalizes the results of Bernstein \cite{bernstein30,bernstein31}, who studied the asymptotic behavior of $\|wT_n^{[-1,1],w}\|_{[-1,1]}$ for general weight functions $w:[-1,1]\rightarrow [0,\infty)$.

\section{Chebyshev polynomials on lemniscatic arcs}
\label{sec:chebyshev_polynomials_system_of_arcs}
To illustrate the applicability of Theorem \ref{thm:main_norm_asymptotics}, we derive asymptotic formulas for the Widom factors associated with a family of unions of arcs. These results complement the analysis in \cite{christiansen-eichinger-rubin23}. 

As previously mentioned, when $\Gamma$ is a smooth arc,
\[
\limsup_{n \to \infty} \cW_n(\Gamma, w) \leq 2 \exp\left(\int_{\Gamma} \log w(x) \, d\mu_\Gamma(x)\right),
\] 
with equality if and only if $\Gamma$ is a straight line segment. This result is rooted in the fact that a straight line segment is the only Jordan arc \(\Gamma\) satisfying 
\begin{equation} 
\frac{\partial g_{\Omega_\Gamma}}{\partial n_+}(z,\infty) = \frac{\partial g_{\Omega_\Gamma}}{\partial n_-}(z,\infty), \quad \Omega_\Gamma = \RS\setminus \Gamma
\label{eq:s-property-arc}
\end{equation} 
at all interior points. The generalization of \eqref{eq:s-property-arc} to a system of arcs is given by the $S$-property, which is formally defined below in Definition \ref{def:s-property}.

The results of \cite{christiansen-eichinger-rubin23} suggest that the $S$-property is the key mechanism behind large limits of Widom factors, specifically leading to a limit of 2. In this section, we provide further evidence for this by examining a family of unions of arcs where the $S$-property fails and show that their corresponding Widom factors are asymptotically strictly less than 2.

\subsection{Main Result}
Extremal polynomials on lemniscates have been extensively studied, particularly for sets of the form $\{z:z^m+1\in r\T\}$ (see e.g. \cite{bergman-rubin24, peherstorfer-steinbauer01, minadiaz06, ullman60, he94}). To obtain a system of arcs from this structure, we consider the pre-image of a circular arc rather than a full circle:
\begin{equation}
	\E_{m,r}(\alpha):=\{z: z^m+1\in r \Gamma_\alpha\},\quad 0<\alpha<\pi,
	\label{eq:def_lemniscatic_arc}
\end{equation}
where $m\in \N$ and $r>0$. The geometry of these sets depends critically on the parameter $r$. The sets in \eqref{eq:def_lemniscatic_arc} consist of \(m\) disjoint analytic arcs when \(r \neq 1\). However, if \(r = 1\), the components connect at the origin, where \(2m\) analytic arcs meet. Examples of these structures are shown in Figure \ref{fig:system_arcs}.

\begin{figure}[htbp]
    \centering
    \begin{subfigure}[t]{0.32\textwidth}
        \centering
        \includegraphics[width=0.8\textwidth]{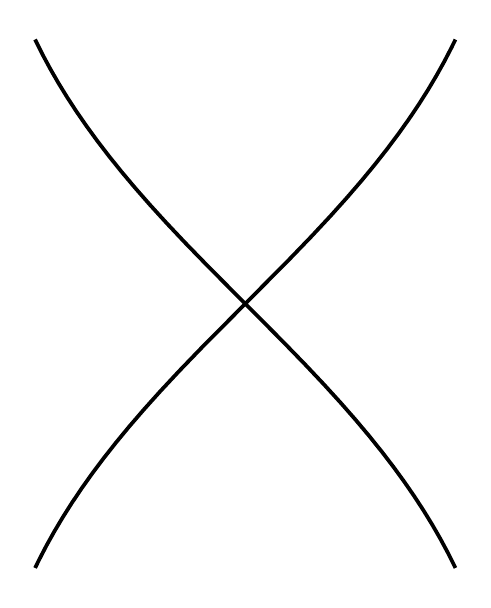} 
        \caption{\(\E_{2,1}(\pi/7)\)}
    \end{subfigure}
    \hfill
    \begin{subfigure}[t]{0.32\textwidth}
        \centering
        \includegraphics[width=\textwidth]{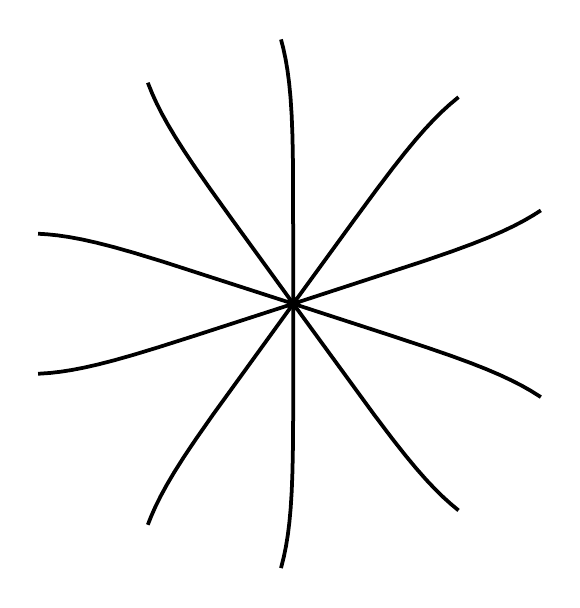} 
        \caption{\(\E_{5,1}(\pi/7)\)}
    \end{subfigure}
    \hfill
    \begin{subfigure}[t]{0.32\textwidth}
        \centering
        \includegraphics[width=\textwidth]{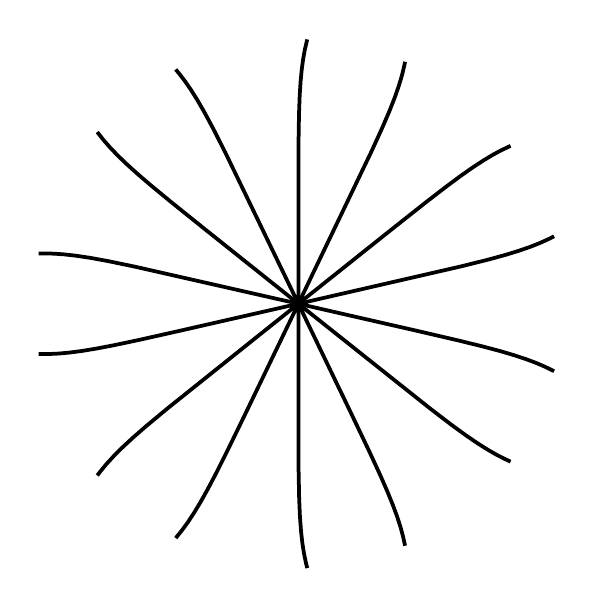} 
        \caption{\(\E_{7,1}(\pi/7)\)}
    \end{subfigure}
    \begin{subfigure}[t]{0.32\textwidth}
        \centering
        \includegraphics[width=0.8\textwidth]{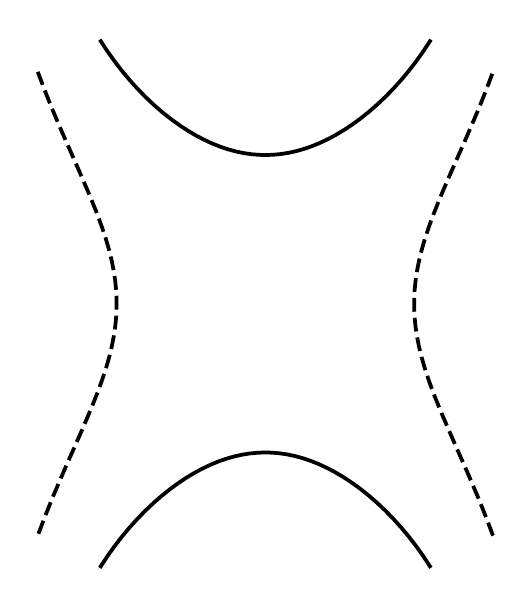} 
        \caption{\(\E_{2,0.9}(\pi/7)\) and \(\E_{2,1.1}(\pi/7)\)}
    \end{subfigure}
    \hfill
    \begin{subfigure}[t]{0.32\textwidth}
        \centering
        \includegraphics[width=\textwidth]{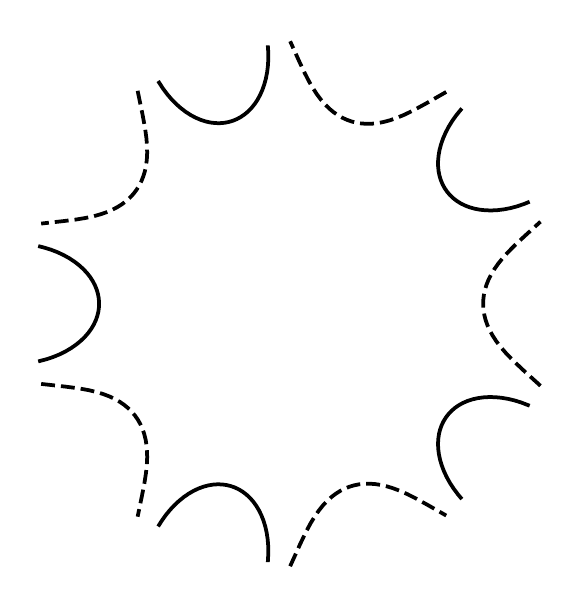} 
        \caption{\(\E_{5,0.9}(\pi/7)\) and \(\E_{5,1.1}(\pi/7)\)}
    \end{subfigure}
    \hfill
    \begin{subfigure}[t]{0.32\textwidth}
        \centering
        \includegraphics[width=\textwidth]{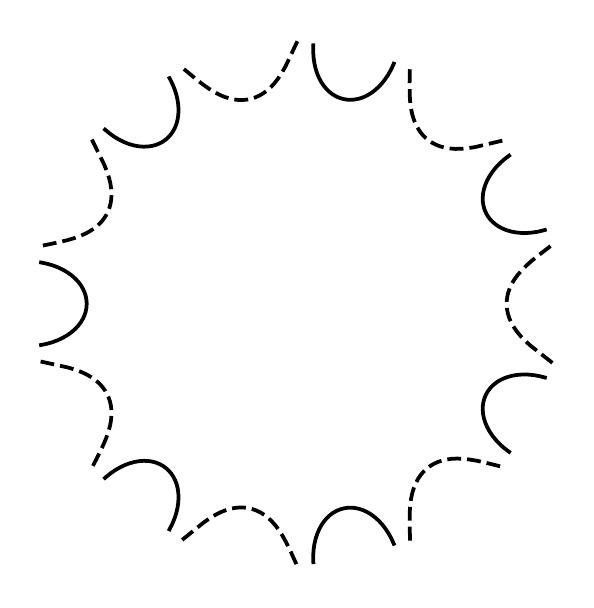} 
        \caption{\(\E_{7,0.9}(\pi/7)\) and \(\E_{7,1.1}(\pi/7)\)}
    \end{subfigure}

    \caption{Examples of sets in the families $\E_{2,r}(\pi/7)$, $\E_{5,r}(\pi/7)$ and $\E_{7,r}(\pi/7)$ for varying values of $r$. The lines corresponding to \(r=1.1\) are dashed.}
    \label{fig:system_arcs}
\end{figure}

Our main result provides the explicit limit of the Widom factors for this family of sets.

\begin{theorem} 	
\label{thm:Widom_factors_lemniscatic_arcs}
	Let $\E_{m,r}(\alpha)=\{z: z^m+1\in r \Gamma_\alpha\}$, where $m\in \N$, $r>0$, and $0<\alpha<\pi$. Then, for any $0\leq l <m$, we have
\begin{equation}
   \lim_{n\rightarrow \infty}\cW_{nm+l}\bigl(\E_{m,r}(\alpha)\bigr) = 2\cos^2(\alpha/4)\, c(r,\alpha)^{l/m}
\end{equation} 	
	where
\begin{equation}
\label{eq:c_r_alpha_definition}
c(r,\alpha)=\frac{\vert 1-r \vert + \sqrt{1-2r\cos(\alpha)+r^2}}{2r\sin(\alpha/2)}.
\end{equation}
	In particular, for the connected case where $r = 1$, we have
\begin{equation}
\label{eq:widom_limit_r_1}
\lim_{n\rightarrow \infty}\cW_n\bigl(\E_{m,1}(\alpha)\bigr) = 2\cos^2(\alpha/4).
\end{equation}
\end{theorem}

This theorem is significant because, as we will show in Proposition \ref{prop:no_s_property}, the connected set $\E_{m,1}(\alpha)$ does not satisfy the $S$-property. The explicit formula \eqref{eq:widom_limit_r_1} demonstrates that this absence of the $S$-property leads to a limit strictly less than 2, supporting the conjecture that the $S$-property is essential for Widom factors of connected sets to approach 2.

\subsection{Proof of Theorem \ref{thm:Widom_factors_lemniscatic_arcs}}
The proof proceeds in several steps. First, we formally define the $S$-property and prove that the set $\E_{m,1}(\alpha)$ does not possess it. Second, we use symmetry arguments to reduce the problem to an analysis of weighted Chebyshev polynomials on the circular arc $\Gamma_\alpha$. Finally, we compute the necessary potential-theoretic quantities to derive the asymptotic formula.

We begin with the formal definition.
\begin{definition}
\label{def:s-property}
    Let $\E\subset \C$ be a compact set with $\Cap(\E)>0$, and suppose that $\Omega_\E=\RS\setminus \E$ is connected. Additionally, assume there exists a subset $\E_0\subset \E$ with $\Cap(\E_0) = 0$, such that
    \begin{equation}
    	\E\setminus \E_0 = \bigcup_{i\in I}\Gamma_i,
    \end{equation}
    where the $\Gamma_i$'s are disjoint open analytic Jordan arcs, and $I\subset \N$. We say that $\E$ satisfies the \emph{$S$-property} if
    \begin{equation}
		\frac{\partial g_{\Omega_\E}}{\partial n_+}(z,\infty)=\frac{\partial g_{\Omega_\E}}{\partial n_-}(z,\infty)
		\label{eq:s-property}
	\end{equation}
    holds for all $z\in \E\setminus \E_0$, where $n_+$ and $n_-$ denote the unit normals from each side of the arcs.
\end{definition}

The \(S\)-property is closely related to capacity minimization, and we refer the reader to \cite{stahl85-1,stahl85-2,stahl12} for further details on this. In \cite{christiansen-eichinger-rubin23}, the Chebyshev polynomials associated with the $m$-stars defined by $\E_m:=\{z:z^m\in [-2,2]\}$ are studied. These $m$-stars serve as examples of sets with the $S$-property, and it is shown that $\cW_n(\E_m)\rightarrow 2$ as $n\rightarrow \infty$ for any $m$. This behavior is analogous to that of Jordan arcs satisfying the $S$-property (i.e., straight line segments).

In order to show that $\E_{m,1}(\alpha)$ does not satisfy \eqref{eq:s-property} at interior points of the subarcs, we utilize the connection to Chebotarev sets. Given $N$ distinct points $a_1\dotsc,a_N$ in the complex plane, there is a unique continuum (i.e., a compact connected set) of minimal logarithmic capacity containing these points. This set is known as the Chebotarev set for $\{a_1,\ldots,a_N\}$, with relevant references detailing its properties found in \cite{stahl12, grotzsch30, schiefermayr14, schiefermayr15, kuzmina80, fedorov85, martinez-finkelshtein-rakhmanov11, ortega-cerda-pridhnani08}. 
It is known (see \cite{stahl12}) that every Chebotarev set necessarily has the form specified in Definition \ref{def:s-property} and therefore satisfies the $S$-property. Conversely, a compact set with the specified structure of analytic arcs that satisfies the $S$-property must be the Chebotarev set for some finite collection of points, as can be inferred from the discussion following Problem 2 and Theorem 11 in \cite{stahl12}.
In our case, this implies the following: if $\E_{m,1}(\alpha)$ were to satisfy \eqref{eq:s-property} away from the points \(\{0\}\cup\{z: z^m+1\in e^{\pm i \alpha}\}\), it would necessarily coincide with the Chebotarev set associated with these points, excluding the point at zero.

\begin{proposition}
\label{prop:no_s_property}
For $0<\alpha<\pi$ and $m\in \N$, the set $\E_{m,1}(\alpha)$, as defined in \eqref{eq:def_lemniscatic_arc}, does not satisfy the $S$-property.
\end{proposition}
\begin{proof}
The proof is carried out by showing that $\E_{m,1}(\alpha)$ is \emph{not} the Chebotarev set for the points $\{z:z^m+1\in e^{\pm i\alpha}\}$. 
To see this, let $\mathcal{C}(e^{\pm i\alpha}, 1)$ denote the Chebotarev set for $\{e^{\pm i\alpha}, 1\}$. By \cite[Theorem 1.2]{kuzmina80}, this set is necessarily contained within the convex hull of these points. Since this does not hold for $\Gamma_\alpha$, we conclude that $\Gamma_\alpha\neq \mathcal{C}(e^{\pm i\alpha},1)$, which in turn implies 
\[
\Cap(\Gamma_\alpha)>\Cap\bigl(\mathcal{C}(e^{\pm i\alpha},1)\bigr).
\] 
Furthermore, the set $\{z: z^m+1\in \mathcal{C}(e^{\pm i\alpha}, 1)\}$ forms a continuum (since $\mathcal{C}(e^{\pm i\alpha}, 1)$ contains the critical value of $z^m+1$). Applying \cite[Theorem 5.2.5]{ransford95} twice, we obtain
\[
\Cap\bigl(\{z:z^m+1\in \mathcal{C}(e^{\pm i\alpha},1)\}\bigr) = 
\Cap\bigl(\mathcal{C}(e^{\pm i\alpha},1)\bigr)^{1/m} < 
\Cap(\Gamma_\alpha)^{1/m} = 
\Cap\bigl(\E_{m,1}(\alpha)\bigr).
\]
Thus, \(\E_{m,1}(\alpha)\) is not the minimal capacity set containing \(\{z:z^m+1\in e^{\pm i\alpha}\}\) and consequently does not satisfy the $S$-property.
\end{proof}

Having established that $\E_{m,1}(\alpha)$ lacks the S-property, we now proceed with the main asymptotic calculation for the Widom factors. The argument begins by exploiting the rotational symmetry of the set $\E_{m,r}(\alpha)$. This set is invariant under rotations by $2\pi/m$-radians. As a result, the corresponding Chebyshev polynomial of degree $nm + l$, where $0 \leq l < m$, satisfies the symmetry relation 
\[
T_{nm+l}^{\E_{m,r}(\alpha)}(e^{2\pi i/m}z) = e^{2\pi il/m} T_{nm+l}^{\E_{m,r}(\alpha)}(z).
\]
This implies that $T_{nm+l}^{\E_{m,r}(\alpha)}(z) = z^lQ_n(z^m+1)$, where $Q_n$ is a monic polynomial of degree $n$. By applying the change of variables $z = (r\zeta-1)^{1/m}$, we find that $Q_n(r\zeta) = r^n\tilde{Q}_n(\zeta)$, where $\tilde{Q}_n$ is the monic polynomial of degree $n$ that minimizes the norm on $\Gamma_\alpha$ with respect to the weight function $w_{r,m,l}(\zeta) = |r\zeta-1|^{l/m}$. In other words, $\tilde{Q}_n(\zeta) = T_{n}^{\Gamma_\alpha, w_{r,m,l}}(\zeta)=:T_{n}^{w_{r,m,l}}(\zeta)$, and
\[
\|T_{nm+l}^{\E_{m,r}(\alpha)}\|_{\E_{m,r}(\alpha)} = r^n\|w_{r,m,l}T_n^{w_{r,m,l}}\|_{\Gamma_\alpha}.
\]
The asymptotic behavior of the right-hand side is given by Theorem \ref{thm:main_norm_asymptotics}, which yields
\begin{equation}
	\|w_{r,m,l}T_n^{w_{r,m,l}}\|_{\Gamma_\alpha}\sim 2\cos(\alpha/4)^2\Cap(\Gamma_\alpha)^n\exp\left(\int_{\Gamma_\alpha}\log w_{r,m,l}(\zeta)\,d\mu_{\Gamma_\alpha}(\zeta)\right).
	\label{eq:cheb_on_lemniscatic_arcs_weight_asymptotics}
\end{equation}
To apply this formula, we must evaluate the integral in the exponential term, which is given by the following lemma.
\begin{lemma}
\label{lem:logarithmic_integral}
For $r>0$, the logarithmic potential is given by the identity
\begin{equation}
\label{2.9}
\int_{\Gamma_\alpha}\log |r\zeta-1| \, d\mu_{\Gamma_\alpha}(\zeta) = 
\log\left(\frac{\vert 1-r \vert + \sqrt{1-2r\cos(\alpha)+r^2}}{2}\right).
\end{equation}
\end{lemma}

\begin{proof}
%
The case $r=1$ is a direct consequence of Frostman's theorem \cite[Theorem 3.3.4]{ransford95}, noting that $\Cap(\Gamma_\alpha) = \sin(\alpha/2)$. 

For $r\neq 1$, the integral equals $h_r(\infty)$, where $h_r$ is the solution to the Dirichlet problem on $\Omega_\alpha = \overline{\C}\setminus\Gamma_\alpha$ with boundary data $\log|r\zeta-1|$. From potential theory, 
we have
\begin{equation}
\label{eq:hr_infinity_green_concise}
h_r(\infty) = \log r + \log\Cap(\Gamma_\alpha) + g_{\Omega_\alpha}(1/r,\infty),
\end{equation}
where $g_{\Omega_\alpha}(z,\infty)$ denotes the Green's function for $\Omega_\alpha$ with pole at infinity. Recalling that the function
\[
f(z) = \frac{1}{2} \Bigl( z-1+\sqrt{ (z-e^{i\alpha})(z-e^{-i\alpha}) } \,\Bigr),
\]
with the branch choice of $\sqrt{\cdot}$ such that $f(z) = z + O(1)$ as $z\to\infty$, maps $\Omega_\alpha$ conformally onto $\{w : |w| > \sin(\alpha/2)\}$ (cf. \cite[Exercise 5.2.4]{ransford95}), the Green's function is given by $g_{\Omega_\alpha}(z, \infty) = \log\left({|f(z)|}/{\sin(\alpha/2)}\right)$.
Evaluating at $z=1/r$ yields
\begin{equation}
\label{2.10}
g_{\Omega_\alpha}(1/r, \infty) 
= \log\left( \frac{\vert 1-r \vert + \sqrt{1-2r\cos(\alpha)+r^2}}{2r\sin(\alpha/2)} \right).
\end{equation}
Substituting this and $\Cap(\Gamma_\alpha)=\sin(\alpha/2)$ into \eqref{eq:hr_infinity_green_concise} 
proves the identity stated in the lemma. The continuity of the expression at $r=1$ confirms consistency.
%
\end{proof}

It follows directly from \eqref{2.10} that the quantity $c(r,\alpha)$ defined in \eqref{eq:c_r_alpha_definition} is nothing but the exponential of $g_{\Omega_\alpha}(1/r, \infty)$. 
We now have all the necessary components to complete the proof of Theorem \ref{thm:Widom_factors_lemniscatic_arcs}.
Using the capacity formula $\Cap(\E_{m,r}(\alpha)) = (r\Cap(\Gamma_\alpha))^{1/m}$, we have
\[
	\cW_{nm+l}\bigl(\E_{m,r}(\alpha)\bigr) = 
	\frac{\|T_{nm+l}^{\E_{m,r}(\alpha)}\|_{\E_{m,r}(\alpha)}}{\Cap\bigl(\E_{m,r}(\alpha)\bigr)^{nm+l}} = 
	\frac{r^n\|w_{r,m,l}T_{n}^{w_{r,m,l}}\|_{\Gamma_\alpha}}{\bigl(r\Cap(\Gamma_\alpha)\bigr)^{(nm+l)/m}}.
\]
Substituting the asymptotic formula \eqref{eq:cheb_on_lemniscatic_arcs_weight_asymptotics} and the result of Lemma \ref{lem:logarithmic_integral}, and taking the limit as $n\to\infty$, we get
\begin{align*}
\lim_{n\to\infty} \cW_{nm+l}\bigl(\E_{m,r}(\alpha)\bigr) &= \lim_{n\to\infty} \frac{r^n \cdot 2\cos^2(\alpha/4)\Cap(\Gamma_\alpha)^n \exp\left(\frac{l}{m}\int_{\Gamma_\alpha}\log|r\zeta-1|\,d\mu_{\Gamma_\alpha}(\zeta)\right)}{r^{n+l/m}\Cap(\Gamma_\alpha)^{n+l/m}} \\
&= \frac{2\cos^2(\alpha/4)}{r^{l/m}\Cap(\Gamma_\alpha)^{l/m}} \bigl( r \sin(\alpha/2)\,c(r,\alpha) \bigr)^{l/m}.
\end{align*}
Since $\Cap(\Gamma_\alpha) = \sin(\alpha/2)$, the expression simplifies to $2\cos^2(\alpha/4)c(r,\alpha)^{l/m}$,
which is the statement of the theorem. \qed

\subsection{Analysis of the Limit Points}
It is illuminating to examine the range of the limit points for different values of $r$. The key is to understand the behavior of the function $r\mapsto c(r,\alpha)$ for a fixed $\alpha\in (0,\pi)$. Since $c(r,\alpha)>0$, the sign of $\partial_rc(r,\alpha)$ is determined by its logarithmic derivative. A calculation shows that the derivative vanishes only if $r \in \{0,1\}$ or $\alpha = \pi k$ for $k\in\Z$. In our setting, the derivative is continuous and does not vanish on $(0,1)$ or $(1,\infty)$, so it must have a constant sign on each of these intervals. We conclude that $\partial_r c(r,\alpha)$ is negative for $0<r<1$ and positive for $r>1$. 
Moreover, since 
\[
   \lim_{r\rightarrow 0}c(r,\alpha) = \infty
   \quad \mbox{and} \quad
   \lim_{r\rightarrow \infty}c(r,\alpha) = \frac{1}{\sin(\alpha/2)},
\] 
it follows that $c(\cdot,\alpha)$ maps $(0,1)$ onto $(1,\infty)$ and $(1,\infty)$ onto $\bigl(1, \csc(\alpha/2)\bigr)$. Consequently, for $0\leq l<m$, the limit points of the Widom factors lie in the ranges:
\[
\lim_{n\rightarrow \infty}\cW_{nm+l}\bigl(\E_{m,r}(\alpha)\bigr) \in \begin{cases}
	[2\cos^2(\alpha/4), 2\cos^2(\alpha/4)\csc^{\frac{l}{m}}(\alpha/2)], & r> 1, \vspace{.3cm}\\
	\qquad\qquad\quad [2\cos^2(\alpha/4),\infty), & 0<r<1.
\end{cases}
\]

A key consequence of this analysis is the distinction in convergence behavior between the connected and disconnected cases. When $r=1$, we have $c(1,\alpha)=1$, which makes the limit in Theorem \ref{thm:Widom_factors_lemniscatic_arcs} independent of the subsequence index $l$. This implies that the full sequence $\{\cW_n(\E_{m,1}(\alpha))\}$ converges. In contrast, for $r\neq 1$, the limit depends on the residue $l \equiv n \pmod m$. The full sequence therefore does not converge; instead, it possesses $m$ distinct limit points. This type of behavior, often called limit periodicity, is characteristic of symmetric systems where each disconnected component carries the same harmonic measure as viewed from infinity. 

As a final observation, Theorem \ref{thm:Widom_factors_lemniscatic_arcs} shows that for the connected case, the limit of $\cW_n(\E_{m,1}(\alpha))$ always lies within the interval $(1,2)$. As $\alpha\rightarrow 0$, the set $\E_{m,1}(\alpha)$ increasingly resembles a set of intersecting straight lines, and correspondingly, the limit of its Widom factors approaches $2$. This behavior should be compared with the results of \cite[Theorem 1.2]{christiansen-eichinger-rubin23}.

\section{Reciprocals of polynomials as weight functions}
\label{sec:reciprocal_weight}

We now turn to the analysis of \( R_{n}^{\Gamma_\alpha,w}(u,u_0) \) for rational weights \( w \), with the goal of determining its locally uniform asymptotic behavior on \( \Omega_\alpha = \overline{\mathbb{C}} \setminus \Gamma_\alpha \).
Given a proper subdomain $\Omega$ of $\overline{\C}$ and a point $\zeta\in\Omega$, we define
\begin{equation}
   b_{\Omega}(\cdot, \zeta) := \exp\Bigl[-\bigl(g_{\Omega}(\cdot, \zeta) 
   + i \widetilde{g}_{\Omega}(\cdot, \zeta)\bigr)\Bigr],
\end{equation}
where \( \widetilde{g}_{\Omega}(\cdot, \zeta) \) is a harmonic conjugate of \( g_{\Omega}(\cdot, \zeta) \). In general, this function is multi-valued, analytic, and character-automorphic on $\Omega$. However, if $\Omega$ is simply connected, it defines a Riemann map from $\Omega$ onto the unit disk $\D$, sending $\zeta$ to $0$. This map is unique up to multiplication by a unimodular constant, and we shall typically fix normalization by requiring that $b_\Omega(z, \zeta)>0$ at a designated point $z\in\Omega$. 

Specializing to $\Gamma_\alpha$, we investigate the asymptotic behavior of
$b_{\Omega_\alpha}(u, \infty)^n R_{n}^{\Gamma_\alpha, w}(u, u_0)$. To ensure a consistent normalization, we impose the condition 
\begin{equation}
\label{normal}
\begin{cases}
\displaystyle{
   \; \quad b_{\Omega_\alpha}(u_0, \infty) > 0} \; \mbox{ if $u_0\neq \infty$, and } \smallskip \\ 
\displaystyle{
   \, \lim_{u\to\infty} ub_{\Omega_\alpha}(u, \infty)>0 \; \mbox{ when } u_0=\infty}.
\end{cases}
\end{equation}
With this normalization, the limit in \eqref{normal}, which can be interpreted as the derivative of $b_{\Omega_\alpha}(u, \infty)$ at $u=\infty$, corresponds exactly to the capacity of $\Gamma_\alpha$. Furthermore, near $\infty$ we have the asymptotic expression
\begin{equation}
\label{near inf}
b_{\Omega_\alpha}(u, \infty)=\frac{\Cap(\Gamma_\alpha)}{u}+\mathcal{O}(\vert u\vert^{-2}).
\end{equation}

\subsection{Residual polynomials relative to the origin}
\label{subsec:res_origin}
We begin by considering the case $u_0 =0$. A key step in our analysis is transferring the space of rational functions associated with $\Omega_\alpha $ to those associated with 
\begin{equation}
	\Omega_0 := \RS \setminus \sfA_0, \quad \sfA_0 = \bigl( \R \cup \{ \infty \} \bigr) \setminus (-1,1).
\end{equation}
This transformation will allow us to apply fundamental results by Yuditskii \cite{yuditskii99}. 
We denote the variable corresponding to $\Omega_\alpha$ by $u$ and the variable associated with $\Omega_0$ by $z$. 

The connection between these domains is established through the following change of variables. Let $u:\Omega_0\rightarrow \Omega_\alpha$ be the M\"{o}bius transformation
\begin{equation}
	u(z) = \frac{z-z_0}{z-\overline{z}_0},\quad z_0 := i\tan(\alpha/2).
	\label{eq:mobius_transformation}
\end{equation}
This 
conformal mapping enables us to pass from $z\in \Omega_0$ to $u:=u(z)\in \Omega_\alpha$ while ensuring that $u(\pm 1) = e^{\mp i\alpha}$. Moreover, since $u(\overline{z}_0) = \infty$, we introduce the notation $z_\infty := \overline{z}_0$ to emphasize this property. 
\begin{center}
\begin{tikzpicture}[scale=1.7]

\begin{scope}
    \draw[->] (-2,0) -- (2.5,0) node[right] {$\text{Re}(z)$};
    \draw[->] (0,-1.5) -- (0,1.5) node[above] {$\text{Im}(z)$};
    
    \draw[dashed, very thick] (-1,0) -- (1,0);
    \draw[very thick] (-2.5,0) -- (-1,0);
    \draw[very thick] (1,0) -- (2.5,0);

    \draw[thick,fill=white] (-1,0) circle(0.05);
    \draw[thick,fill=white] (1,0) circle(0.05);

    \node[below] at (-1,-0.1) {$-1$};
    \node[below] at (1,-0.1) {$1$};
    \draw[thick,fill=black] (0,0.5) circle(0.03) node[right] {$z_0 = i\tan\frac{\alpha}{2}$};
    \node at (1,1.4) {$\Omega_0$};
\end{scope}

\begin{scope}[xshift=5cm]
    \draw[dashed, very thick] (1,0) arc[start angle=0, end angle=360, radius=1];
    \draw[very thick] (1,0) arc[start angle=0, end angle=30, radius=1];
    \draw[very thick] (1,0) arc[start angle=0, end angle=-30, radius=1];

    \draw[->] (-1.5,0) -- (1.5,0) node[right] {$\text{Re}(u)$};
    \draw[->] (0,-1.5) -- (0,1.5) node[above] {$\text{Im}(u)$};
	\draw[thick,fill=white] (0.866,0.5) circle(0.05)node[above right] {$e^{i\alpha}$};
    \draw[thick,fill=white] (0.866,-0.5) circle(0.05)node[below right] {$e^{-i\alpha}$};
    \draw[thick,fill=black] (0,0) circle(0.03) node[above left] {$u(z_0)$};

    \node[below right] at (1,0) {$1$};
	
    \node at (1,1.4) {$\Omega_\alpha$};
\end{scope}

\draw[->, thick] (2,0.5) to[out=30, in=150] (4,0.5);

\node[above] at (2.9,0.9) {$u(z)$};
\end{tikzpicture}
\end{center}

When $z,\zeta\in \Omega_0$, the function $u$ satisfies
\begin{equation}
		u(z)-u(\zeta) = \frac{z_0-z_\infty}{\zeta-z_\infty}\frac{z-\zeta}{z-z_\infty}.
		\label{eq:change_of_variables}
\end{equation}
Thus, if $H:\Omega_\alpha\rightarrow \RS$ is a rational function in the variable $u$, then $H(u(z))$ remains a rational function on $\Omega_0$. Specifically, if $H(u) = \prod_{j=1}^{n}(u-u(z_j))^{k_j}$ for some $k_j\in \Z$ and $z_j\in\Omega_0$, then
\begin{equation}
	H\bigl(u(z)\bigr)=c (z-z_\infty)^{-\sum_{j=1}^{n}k_j} \prod_{j=1}^{n}(z-z_j)^{k_j}
	\label{eq:rational_funct_transform}
\end{equation}
for some explicitly computable constant $c$.
This shows that the change of variables $z\mapsto u(z)$ preserves rational functions. 
Our motivation for using this transformation comes from the theory of residual polynomials relative to $\Omega_0$, as developed in \cite{yuditskii99}, which we can now apply.

We will initially restrict our study to weight functions $w$ of the form
\begin{equation}
	w(u) = e^{i\beta}\Biggl[\,\prod_{j=1}^{k}(u-u_j)\Biggr]^{-1},
	\label{eq:weight_function_form}
\end{equation}
where $|u_j|> 1$ and $\beta\in [0,2\pi)$ is chosen so that $w(0)>0$. These functions generally take complex values on $\Gamma_\alpha$ and, strictly speaking, do not qualify as weights. However, since \[\|Pw\|_{\Gamma_\alpha} = \|P|w|\|_{\Gamma_\alpha},\] we use the notation $R_n^{\Gamma_\alpha,w}$ to refer to $R_n^{\Gamma_\alpha, |w|}$. This convention will be adopted throughout without further clarification.

Every $u$ with $|u|>1$ corresponds to a point in the lower half-plane via the mapping $z\mapsto u(z)$. Thus, for each $u_j$ in \eqref{eq:weight_function_form}, there exists $z_j$ with $\Im z_j>0$ such that $u_j = u(\overline{z}_j)$. 
For convenience, we introduce the weight function
\begin{equation}
	W_n(z) = e^{i\theta}(z-z_\infty)^{k-n}\bigl[(z-\overline{z}_1)\cdots (z-\overline{z}_k)\bigr]^{-1},
	\label{eq:transfered_weight}
\end{equation}
where $\theta\in [0,2\pi)$ is chosen so that $W_n(z_0)>0$. When $P \in \cP_n$, it follows from \eqref{eq:rational_funct_transform} that
\[ Q(z) := P\bigl(u(z)\bigr)\frac{w\bigl(u(z)\bigr)}{W_n(z)} \]
defines a polynomial $Q \in \cP_n$. Moreover, this correspondence establishes a bijection between $\cP_n$ and itself, relating the weighted polynomial $Pw$ on $\Gamma_\alpha$ to $QW_n$ on $\sfA_0$.

To study residual polynomials associated with $\sfA_0$ and the weight function $W_n$, we note that the non-compactness of $\sfA_0$ poses challenges in applying classical results. However, if 
\[Q(z) = a_nz^{n}+\text{lower order terms},\]
then
\[\lim_{z\rightarrow \infty}Q(z)W_n(z) = a_ne^{i\theta},\]
ensuring that $QW_n$ extends continuously to $\sfA_0$. A slight modification of the proof of \cite[Theorem 2.1]{christiansen-simon-zinchenko-V} therefore establishes the existence of a unique polynomial $R_n^{\sfA_0,W_n}(\cdot,z_0)$ of degree at most $n$ satisfying
\begin{align*} 
   (i)& \quad 
    \|R_n^{\sfA_0,W_n}(\cdot,z_0)W_n(\cdot)\|_{\sfA_0} \leq 1, 
   \mbox{ and }  \\
   (ii)& \quad R_n^{\sfA_0,W_n}(z,z_0)= \sup\bigl\{|Q(z_0)|: Q\in \cP_n,\, \left\|QW_n\right\|_{\sfA_0}\leq 1\bigr\}.
\end{align*}
The extremal problems on 
$\Gamma_\alpha$ and $\sfA_0$ are related through the following result; see also \cite[Lemma 2.1]{eichinger17}.
\begin{lemma}
\label{lem:identification}
	Let $z_0 = i\tan(\alpha/2)$ and define $u:\Omega_0\rightarrow \Omega_\alpha$ as in \eqref{eq:mobius_transformation}. Then, for all $z\in \Omega_0$,
	\begin{equation}
	   R_{n}^{\Gamma_\alpha, w}\bigl(u(z), 0\bigr)w\bigl(u(z)\bigr)= R_{n}^{\sfA_0, W_n}(z,z_0)W_n(z). 
	\end{equation}	
\end{lemma}
\begin{proof}
	As noted above, the mapping
	\[\cP_n\ni P\mapsto Q(z):=P\bigl(u(z)\bigr)\frac{w\bigl(u(z)\bigr)}{W_n(z)}\in \cP_n\] 
	is a bijection. Furthermore, we have	
	\[\|Pw\|_{\Gamma_\alpha} = \|QW_n\|_{\sfA_0}\]
	and
	\[P(0) = Q(z_0)W_n(z_0)/w(0).\]
	The result now follows from the uniqueness of residual polynomials, together with the facts that $W_n(z_0)>0$ and $w(0)>0$.
\end{proof}

Instrumental to our approach is the fact that $R_{n}^{\sfA_0,W_n}(\cdot, z_0)$ has an explicit representation, as is detailed and proven in \cite{yuditskii99}. A concise formulation of this result is given in \cite[Theorem 2.2]{eichinger17}. There is a unique point $x_n\in (0,1)$ such that, upon defining 
\begin{equation}
\label{I_n}
   I_n := [-x_n,x_n], \quad \Omega_n := \Omega_0\setminus I_n, 
\end{equation}   
the harmonic measure satisfies the identity
\begin{equation}
	(n-k+1)\omega(z_\infty,I_n;\Omega_n)+\sum_{j=1}^{k}\omega(\overline{z}_j,I_n;\Omega_n) = 1.
	\label{eq:sum_harmonic_measures}
\end{equation} 
For this choice of \(x_n\), the function $R_n^{\sfA_0,W_n}(\cdot, z_0)W_n(\cdot)$ admits the representation
\begin{equation} 
	R_n^{\sfA_0,W_n}(z,z_0)W_n(z)=
	\frac{1+s_n(z)}{2B_n^{\Omega_n}(z)}+
	\frac{1-s_n(z)}{2}
	B_n^{\Omega_n}(z) \left(\frac{z-z_0}{z-z_\infty}\right)^{n-k+1} 
	\prod_{j=1}^{k}\frac{z-z_j}{z-\overline{z}_j},
	\label{eq:extremal_solution}
\end{equation}
where
\begin{equation}
	B_n^{\Omega_n}(z) := b_{\Omega_n}(z,z_\infty)^{n-k+1}\prod_{j=1}^{k}b_{\Omega_n}(z,\overline{z} _j)
	\label{eq:blaschke_factors}
\end{equation}
and
\begin{equation}
	s_n(z) := \sqrt{\frac{z_0^2-x_n^2}{z_0^2-1}\frac{z^2-1}{z^2-x_n^2}}.
	\label{eq:s_n_def}
\end{equation}
We impose the normalization conditions $B_{n}^{\Omega_n}(z_0)>0$ and $s_n(z_0) = 1$. 
Although the individual factors in \eqref{eq:blaschke_factors} are not necessarily single-valued, their product is. This property is intrinsically linked to \eqref{eq:sum_harmonic_measures}, as will be explained in the proof of Lemma \ref{Bn}.
To ensure that the square root in \eqref{eq:s_n_def} is single-valued, we introduce branch cuts along $\sfA_0$ and $I_n$. With this choice, the following identities hold:
\[
   \sqrt{\overline{z}^2-1}=-\overline{\sqrt{z^2-1}}
   \; \mbox{ and } \;
   \sqrt{\overline{z}^2-x^2_n}=\overline{\sqrt{z^2-x_n^2}}.
\]
It follows that $s_n(\overline{z})=-\overline{s_n(z)}$ and, in particular, $s_n(z_\infty)=-1$.

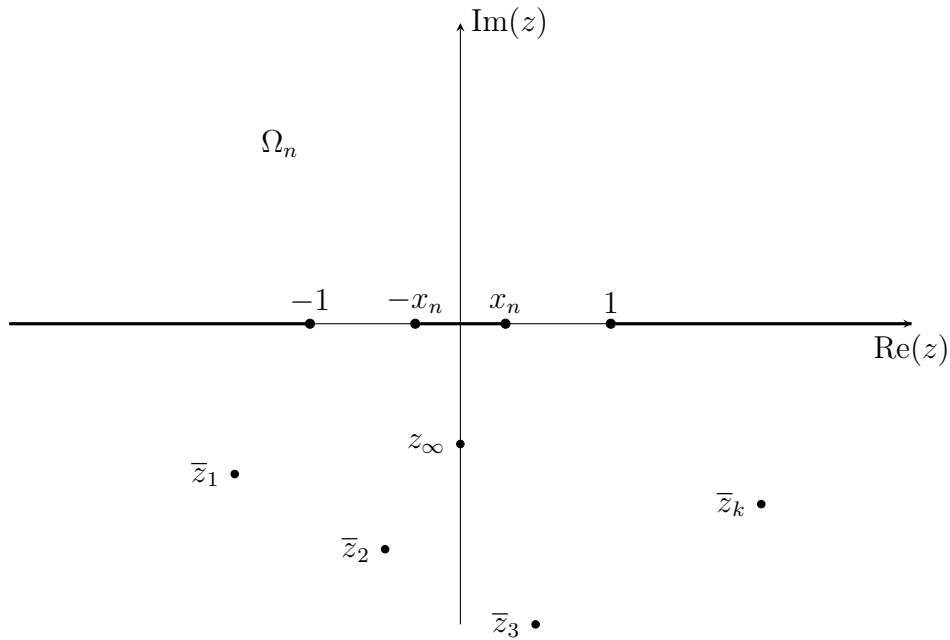
\begin{figure}[h!]
	\centering	
	\begin{tikzpicture}[scale=2, >=stealth]

    \draw[->] (-3, 0) -- (3, 0) node[anchor=north] {Re($z$)};
    \draw[->] (0, -2) -- (0, 2) node[anchor=west] {Im($z$)};
     \node at (-1.2,1.2) {$\Omega_n$};
    
    \draw[very thick, black] (-2.999, 0) -- (-1, 0) ;
    \draw[very thick, black] (-0.3, 0) -- (0.3, 0) ;
    \draw[very thick, black] (1, 0) -- (2.999, 0);
    \draw[fill=black] (-1, 0) circle (0.03cm) node[anchor=south] {$-1$};
    \draw[fill=black] (1, 0) circle (0.03cm) node[anchor=south] {$1$};
    
    \draw[fill=black] (-0.3, 0) circle (0.03cm) node[anchor=south] {$-x_n$};
    \draw[fill=black] (0.3, 0) circle (0.03cm) node[anchor=south] {$x_n$};

    \node[draw, circle, fill=black, inner sep=1pt, label=left:$\overline{z}_1$] at (-1.5, -1) {};
    \node[draw, circle, fill=black, inner sep=1pt, label=left:$\overline{z}_2$] at (-0.5, -1.5) {};
    \node[draw, circle, fill=black, inner sep=1pt, label=left:$\overline{z}_3$] at (0.5, -2) {};
    \node[draw, circle, fill=black, inner sep=1pt, label=left:$\overline{z}_k$] at (2, -1.2) {};
    
     \node[draw, circle, fill=black, inner sep=1pt, label=left:$z_\infty$] at (0, -0.8) {};
    
%
    
	\end{tikzpicture}
	\caption{The domain $\Omega_n$ with designated points $\overline{z}_j$.}
\end{figure}
\begin{lemma} \label{Bn}
The function $B_n^{\Omega_n}$, as defined in \eqref{eq:blaschke_factors}, is single-valued in $\Omega_n$. Moreover, the right-hand side of \eqref{eq:extremal_solution}, when divided by $W_n(z)$, extends from an analytic function on $\Omega_n$ to a polynomial.
\end{lemma}

\begin{proof}
We first show that \(B_n^{\Omega_n}\) is single-valued. This follows directly from \eqref{eq:sum_harmonic_measures}. Since the modulus $|b_{\Omega_n}(\cdot,\zeta)|$ is single-valued, it suffices to show that the change in argument of $B_n^{\Omega_n}$ along any Jordan curve $\gamma \subset \Omega_n$ enclosing $I_n$, denoted $\Delta_{\gamma} \arg B_n^{\Omega_n}$, is an integer multiple of $2\pi$. For any such positively oriented curve $\gamma$, it is known (cf. \cite[\S 5.2, Ch.\;6]{ahlfors78} or \cite[Ch.\;4]{widom69}) that $\Delta_{\gamma} \arg b_{\Omega_n}(\cdot,\zeta)=2\pi\omega(\zeta, I_n;\Omega_n)$. Hence, using \eqref{eq:sum_harmonic_measures}, we find
\begin{align*}
\Delta_{\gamma} \arg B_n^{\Omega_n} &=
\Delta_{\gamma} \arg \biggl(b_{\Omega_n}(\cdot,z_\infty)^{n-k+1}\prod_{j=1}^{k}b_{\Omega_n}(\cdot,\overline{z}_j)\biggr) \\
&= 2\pi\biggl((n-k+1)\omega(z_\infty,I_n;\Omega_n)+\sum_{j=1}^{k}\omega(\overline{z}_j,I_n;\Omega_n)\biggr) = 2\pi.
\end{align*}
Thus, \(B_n^{\Omega_n}\) is single-valued in $\Omega_n$. Consequently, the right-hand side of \eqref{eq:extremal_solution}, which we denote by $R(z)$, is also single-valued in $\Omega_n$.

To prove the second claim, we show that $R$ is a rational function. Since $R/W_n$ has no poles in $\Omega_n$, this establishes the claim. The function $R$ is analytic in $\Omega_n$ except possibly at the zeros of $B_n^{\Omega_n}$, which are poles for $R$. Furthermore, since \(|B_n^{\Omega_n}| = 1\) on \(\partial \Omega_n\), $R$ is bounded near infinity within $\Omega_n$. To prove that $R$ is rational, it remains only to show that $R$ extends continuously across $\partial \Omega_n$. Since
\[ b_{\Omega_n}(z,\zeta)\biggl(\frac{z-\overline{\zeta}}{z-\zeta}\biggr) = b_{\Omega_n}(z,\overline{\zeta}) = \overline{b_{\Omega_n}(\overline{z},\zeta)}, \]
it follows that
\begin{equation}
	B_n^{\Omega_n}(z)\left(\frac{z-z_0}{z-z_\infty}\right)^{n-k+1}
	\prod_{j=1}^{k}\frac{z-z_j}{z-\overline{z}_j} = \overline{B_n^{\Omega_n}(\overline{z})}.
	\label{eq:conjugating_B}
\end{equation}
Since $s_n^2(x) > 0$ for $x \in \partial \Omega_n$, the boundary values $s_n(x\pm i0)$ are real and non-zero. Recalling the symmetry $s_n(z) = -\overline{s_n(\overline{z})}$, we get $s_n(x+i0) = -s_n(x-i0)$ for $x \in \partial \Omega_n$. Using this relation, \eqref{eq:conjugating_B}, and the fact that $|B_n^{\Omega_n}|=1$ on $\partial \Omega_n$, we find for $x \in \partial \Omega_n$:
\begin{align*}
	R(x+i0) &= \frac{1+s_n(x+i0)}{2B_n^{\Omega_n}(x+i0)}+\frac{1-s_n(x+i0)}{2}\overline{B_n^{\Omega_n}(x-i0)} \\
	& = \frac{1-s_n(x-i0)}{2}\overline{B_n^{\Omega_n}(\overline{x-i0})}+\frac{1+s_n(x-i0)}{2B_n^{\Omega_n}(x-i0)}
	= R(x-i0).
\end{align*}
Therefore, $R$ extends continuously across $\partial \Omega_n$, establishing that $R$ is a rational function.
\end{proof}

The product $R_n^{\sfA_0,W_n}(\cdot ,z_0)W_n(\cdot)$ has poles at $z_\infty$ and $\overline{z}_1,\dotsc,\overline{z}_k$, but these are naturally cancelled by multiplication with the factor
\begin{equation}
	B_n^{\Omega_0}(z):=b_{\Omega_0}(z,z_\infty)^{n-k}\prod_{j=1}^{k}b_{\Omega_0}(z,\overline{z}_j),
	\label{eq:B_0}
\end{equation}
where we again choose the normalization so that $B_n^{\Omega_0}(z_0)>0$. Defining the functions $f_n:\Omega_0\rightarrow\C$ by
\begin{equation}
	f_n(z) := B_{n}^{\Omega_0}(z)R_n^{\sfA_0,W_n}(z,z_0)W_n(z),
	\label{eq:f_n_sequence}
\end{equation}
we see that each $f_n$ is analytic throughout $\Omega_0$ and satisfies $f_n(z_0)>0$. Moreover, since $\vert f_n \vert \leq 1$ on $\partial\Omega_0 = \sfA_0$, the maximum principle implies that $\|f_n\|_{\Omega_0}\leq 1$ for all $n$. By Montel's theorem \cite[Theorem VII.2.9]{conway78}, any such sequence has a uniformly convergent subsequence on compact subsets of $\Omega_0$, converging to an analytic function. We will proceed by showing that $\{f_n\}$ has a unique limit point, implying that the full sequence converges locally uniformly.

\begin{proposition}
		Let $\{f_n\}$ be defined by \eqref{eq:f_n_sequence}. Then
	\begin{equation} \label{limit f_n}
	   \lim_{n\rightarrow \infty}f_n(z)= \frac{1+s(z)}{2}\frac{b_{\Omega_0}(z,0)}{b_{\Omega_0}(z,z_\infty)}
	\end{equation}
	uniformly on compact subsets of $\Omega_0$, where
	\begin{equation}
		s(z) := \frac{z_0 \sqrt{z^2-1}}{z \sqrt{z_0^2-1}}.
		\label{eq:s_definition}
	\end{equation}
	We adopt the normalization $s(z_0)=1$ and require that $b_{\Omega_0}(z_0, 0)$ and $b_{\Omega_0}(z_0, z_\infty)$ are positive.
	\label{prop:asymptotics_of_real_problem}	
\end{proposition}

We will establish this convergence in several steps, closely following the approach in \cite{eichinger17} while introducing the necessary modifications to account for the presence of a weight function. 

\begin{lemma} \label{lem:intervals_shrinking}
	Let $I_n = [-x_n,x_n]$ be defined as in \eqref{I_n}. Then $x_n\rightarrow 0$ as $n\rightarrow \infty$. Moreover, the restriction of $n \omega(z_\infty,dx;\Omega_n)$ to $I_n$ converges in the weak-star sense to the Dirac measure at the origin:
	\begin{equation} \label{d0}
	   n\omega(z_\infty,dx; \Omega_n)\big\vert_{I_n}\xrightarrow{\enskip \ast\enskip}\delta_0(dx)
	\; \text{ as } n\rightarrow \infty.
	\end{equation}
\end{lemma}

\begin{proof}
Let $\omega_n(z, I) := \omega(z, I; \Omega_n)$. The functions $z \mapsto \omega_n(z, I_n)$ are positive harmonic functions in the lower half-plane. By Harnack's inequality \cite[Corollary 1.3.3]{ransford95}, there exists a constant $M \ge 1$, independent of $n$, such that
\begin{equation}
	M^{-1}\omega_n(z_\infty, I_n) \le \omega_n(\overline{z}_j, I_n) \le M\omega_n(z_\infty, I_n), \quad j=1,\dotsc,k.
	\label{eq:harnack}
\end{equation}
Combining \eqref{eq:harnack} with the identity \eqref{eq:sum_harmonic_measures} yields
\begin{align*}
	1 = (n-k+1)\omega_n(z_\infty, I_n) + \sum_{j=1}^{k}\omega_n(\overline{z}_j, I_n) 
	\ge \left(n-k+1 + k M^{-1}\right)\omega_n(z_\infty, I_n).
\end{align*}
This implies $\omega_n(z_\infty, I_n) \le (n-k+1 + kM^{-1})^{-1}$, and thus
\begin{equation} \label{eq:limit_harmonic_measure}
 \lim_{n\to\infty} \omega_n(z_\infty, I_n) = 0.
\end{equation}

To show that $x_n \to 0$, we define the function $f(x) := \omega(z_\infty, [-x, x]; \Omega_x)$ for $0<x<1$, where $\Omega_x := \Omega_0 \setminus [-x, x]$. When $a \leq b$, we have $\Omega_b \subset \Omega_a$, and by the domain monotonicity principle for harmonic measure \cite[Corollary 4.3.9]{ransford95}, it follows that $\omega(z_\infty, \sfA_0; \Omega_b) \le \omega(z_\infty, \sfA_0; \Omega_a)$. Since $\omega(z_\infty, [-x, x]; \Omega_x) = 1 - \omega(z_\infty, \sfA_0; \Omega_x)$, this yields $f(a) = \omega(z_\infty, [-a, a]; \Omega_a) \le \omega(z_\infty, [-b, b]; \Omega_b) = f(b)$. Thus, $f(x)$ is a non-decreasing function of $x$. Since $f$ is positive and non-decreasing, the limit $\lim_{n \to \infty} f(x_n) = 0$, shown in \eqref{eq:limit_harmonic_measure}, implies that $x_n \to 0$.

Finally, combining the upper bound for $\omega_n(z_\infty, I_n)$ derived from \eqref{eq:harnack} with the corresponding lower bound, we find
\[
\frac{1}{n-k+1 + kM} \le \omega_n(z_\infty, I_n) \le \frac{1}{n-k+1 + kM^{-1}}.
\]
This implies $\lim_{n\to\infty} n \omega_n(z_\infty, I_n) = 1$. Let $\nu_n := n \omega_n(z_\infty, dx)|_{I_n}$. Then $\{\nu_n\}$ is a sequence of probability measures with $\mathrm{supp}(\nu_n) = I_n = [-x_n, x_n]$. Since $x_n \to 0$, the supports shrink to the origin. This, combined with the total mass converging to 1, implies the weak-star convergence $\nu_n \xrightarrow{*} \delta_0$ as $n \to \infty$. 
\end{proof}

The limiting behavior of the functions $s_n$ defined in \eqref{eq:s_n_def} can now be established. 
\begin{lemma}
	\label{lem:limiting_behaviours}
	Let $B_n^{\Omega_n}$, $s_n$, and $B_n^{\Omega_0}$ be defined by \eqref{eq:blaschke_factors}, \eqref{eq:s_n_def}, and \eqref{eq:B_0}, respectively, and let $s$ be defined by \eqref{eq:s_definition}. Then
	\begin{equation}
		\lim_{n\rightarrow\infty}s_n(z)^2 = s(z)^2=\frac{z_0^2}{z_0^2-1}\frac{z^2-1}{z^2}
		\label{eq:s_n_limit}
	\end{equation}
	and
	\begin{equation}
		\lim_{n\rightarrow \infty}B_n^{\Omega_0}(z)\frac{1-s_n(z)}{2}B_n^{\Omega_n}(z)\left(\frac{z-z_0}{z-z_\infty}\right)^{n-k+1}\prod_{j=1}^{k}\frac{z-z_j}{z-\overline{z}_j}=0,\label{eq:vanishing_factor_limit}
	\end{equation}
	uniformly on compact subsets of $\Omega_0\setminus\{0\}$.
\end{lemma}
\begin{proof}
	It is clear that \eqref{eq:s_n_limit} follows directly from Lemma \ref{lem:intervals_shrinking} and the definition of $s_n$.
	To prove \eqref{eq:vanishing_factor_limit}, note that for $j=0,1,\dotsc,k$, the function 
	$\vert b_{\Omega_n}(z,\overline{z}_j){(z-z_j)}/{(z-\overline{z}_j)}\vert$ has boundary values at most $1$ 
	on $\partial\Omega_n$. By the maximum modulus principle, this implies that
	\[\left|b_{\Omega_n}(z,\overline{z}_j)\frac{z-z_j}{z-\overline{z}_j}\right| < 1,
	   \quad z\in \Omega_n, \quad j=0,1,\dotsc, k.\] 
Since $B_n^{\Omega_0}(z)\to 0$ uniformly on every compact subset of $\Omega_0$ as $n\to\infty$, the result follows.
\end{proof}

We are now in position to analyze the limit points of the sequence $\{f_n\}$ defined in \eqref{eq:f_n_sequence} and show that the full sequence converges uniformly on compact subsets of the domain $\Omega_0$.

\begin{proof}[Proof of Proposition \ref{prop:asymptotics_of_real_problem}]
We begin by noting that the right-hand side of \eqref{limit f_n} is analytic in $\Omega_0$. This follows from the fact that the simple pole of $s(z)$ at $z=0$ is canceled by the zero of $b_{\Omega_0}(z,0)$, while the simple pole at $z=z_\infty$ also vanishes due to the identity $s(z_\infty)=-1$.

For $z,\zeta\in \C\setminus\R$, we have 
	\[\log\left|\frac{b_{\Omega_0}(z,\zeta)}{b_{\Omega_n}(z,\zeta)}\right| = g_{\Omega_n}(z,\zeta)-g_{\Omega_0}(z,\zeta) = -\int_{-x_n}^{x_n}g_{\Omega_0}(z,x)\,\omega(\zeta,dx;\Omega_n).\]
	As in the proof of Lemma \ref{lem:intervals_shrinking}, it follows that $\omega(\zeta,I_n;\Omega_n) \rightarrow 0$ as $n\rightarrow \infty$. Together with the boundedness of $x\mapsto g_{\Omega_0}(z,x)$ on $I_n$ for fixed $z$, this implies
	\[\log\left|\frac{b_{\Omega_0}(z,\zeta)}{b_{\Omega_n}(z,\zeta)}\right|\longrightarrow 0 \; \mbox{ as } n\rightarrow \infty.\]
	In particular, for $j=0, 1, \ldots, k$, we obtain 
	\[ |b_{\Omega_n}(z, \overline{z}_j)|\rightarrow |b_{\Omega_0}(z,\zeta)| \; \mbox{ as } n\rightarrow \infty. \]  
	Applying \eqref{d0}, we further deduce that
	\[\log\left|\frac{b_{\Omega_0}(z,z_\infty)}{b_{\Omega_n}(z,z_\infty)}\right|^{n-k} = -\int_{-x_n}^{x_n}g_{\Omega_0}(z,x)(n-k)\omega(z_\infty,dx;\Omega_n)\longrightarrow -g_{\Omega_0}(z,0) \; \text{ as } n\rightarrow \infty.\]	
	By combining this with \eqref{eq:extremal_solution} and Lemma \ref{lem:limiting_behaviours}, we conclude that 
	\[|f_n(z)| = \left|B_n^{\Omega_0}(z)R_n^{\sfA_0,W_n}(z,z_0)W_n(z)\right| \longrightarrow \left|\frac{1+s(z)}{2}\frac{b_{\Omega_0}(z,0)}{b_{\Omega_0}(z,z_\infty)}\right| \; \text{ as }n\rightarrow \infty.
	\]
The limit in \eqref{limit f_n} now follows from the normal family property of $\{ f_n \}$ or directly from \cite[Prop.~4.2]{christiansen-simon-zinchenko-II}, given that $f_n(z_0)>0$ for all $n$ and
\[
   \frac{1+s(z_0)}{2}\frac{b_{\Omega_0}(z_0,0)}{b_{\Omega_0}(z_0,z_\infty)}>0
\]
by our normalization.	
%
\end{proof}

Our focus is on weighted polynomials associated with circular arcs, which leads us to the task of transferring asymptotic limits back to $\Omega_\alpha$. 
Since $u$ is a conformal mapping, 
and given that both $b_{\Omega_0}(z_0, z_\infty)$ and $b_{\Omega_\alpha}(0, \infty)$ are positive, it follows that
\begin{equation}
   b_{\Omega_0}(z, \zeta) = b_{\Omega_\alpha}\bigl(u(z), u(\zeta)\bigr), \quad z, \zeta \in\Omega_0.
\end{equation}
Combining this with Lemma \ref{lem:identification}, we obtain
\begin{equation} \label{bRw}
   b_{\Omega_\alpha}\bigl(u(z),\infty\bigr)^n R_{n}^{\Gamma_\alpha,w}\bigl(u(z),0\bigr)
   w\bigl(u(z)\bigr)
   =b_{\Omega_0}(z,z_\infty)^nR_n^{\sfA_0,W_n}(z,z_0)W_n(z).
\end{equation}  
A direct computation using \eqref{eq:mobius_transformation} and \eqref{eq:rational_funct_transform} yields the expression
\begin{equation}
\label{s}
   s(z) = \frac{\sqrt{\bigl(u(z)-e^{i\alpha}\bigr)\bigl(u(z)-e^{-i\alpha}\bigr)}}{u(z)+1},
\end{equation}
where the branch of the square root is taken with a cut along $\Gamma_\alpha$ and is normalized to be $1$ at $u(z)=0$. 
Applying Proposition \ref{prop:asymptotics_of_real_problem} --- where the core asymptotic analysis is carried out --- and keeping \eqref{eq:B_0}--\eqref{eq:f_n_sequence} in mind, we conclude that as $n\to\infty$, the expression in \eqref{bRw} converges to
\begin{multline} 
    \qquad \frac{1+s(z)}{2} \frac{b_{\Omega_0}(z,0)}{b_{\Omega_0}(z,z_\infty)}
    \prod_{j=1}^{k}\frac{b_{\Omega_0}(z,z_\infty)}{b_{\Omega_0}(z,\overline{z}_j)}  \\
    =\frac{u(z)+1+\sqrt{\bigl(u(z)-e^{i\alpha}\bigr)\bigl(u(z)-e^{-i\alpha}\bigr)}}{2\bigl( u(z)+1 \bigr)}
    \frac{b_{\Omega_\alpha}\bigl(u(z), -1\bigr)}{b_{\Omega_\alpha}\bigl(u(z),\infty\bigr)}
    \prod_{j=1}^{k}\frac{b_{\Omega_\alpha}\bigl(u(z),\infty\bigr)}{b_{\Omega_\alpha}\bigl(u(z), u_j\bigr)}. 
    \qquad
\end{multline}
All that remains is to divide by the weight and simplify. 
The function 
\begin{equation}
   f(u):= \biggl( \prod_{j=1}^{k} \frac{b_{\Omega_\alpha}(u,\infty)}{b_{\Omega_\alpha}(u,u_j)} \biggr) 
   \Big/ w(u)
\end{equation} 
is analytic and non-vanishing in $\Omega_\alpha$, as poles and zeros cancel out. Consequently, $\log\vert f(u)\vert$ is harmonic throughout $\Omega_\alpha$. By \eqref{eq:harmonic_measure_property}, so is the function
\begin{equation}
\label{hw}
   h(u) := -\int_{\Gamma_\alpha}\log |w(x)|\,\omega(u,dx,\Omega_{\alpha}).
\end{equation}
Since both functions attain the same boundary values, $-\log\vert w(x)\vert$ on $\Gamma_\alpha$, they must be identical by the maximum principle. 
Let $\widetilde{h}$ denote the harmonic conjugate of $h$, normalized to vanish at $u_0$. Then the analytic function
\begin{equation}
\label{eq:outer_function}
   F_w(u, u_0) := \exp \Bigl[ h(u)+i \widetilde{h}(u) \Bigr]
\end{equation}
becomes positive at $u=u_0$. Since our normalization ensures that $f(0)>0$, it follows that 
\begin{equation}
   f(u)=F_w(u, 0), \quad u\in\Omega_\alpha.
\end{equation}    
Thus, assembling our results, we arrive at the final asymptotic formula
\begin{equation}
\label{eq:residual_at_zero_limit}
   b_{\Omega_\alpha}(u,\infty)^n R_{n}^{\Gamma_\alpha,w}(u,0) \longrightarrow  
   \frac{u+1+\sqrt{(u-e^{i\alpha})(u-e^{-i\alpha})}}{2(u+1)}
   \frac{b_{\Omega_\alpha}(u, -1)}{b_{\Omega_\alpha}(u,\infty)} F_w(u, 0),
\end{equation}
which holds uniformly on compact subsets of $\Omega_\alpha$ as $n\to\infty$. Notably, the singularity at $u=-1$ is exactly canceled by the zero of $b_{\Omega_\alpha}(u, -1)$.

In the next section, specifically in Theorem \ref{thm:weighted_residual_asymptotics}, we establish an asymptotic formula valid for general \( u_0 \). For now, we present a straightforward approach to obtaining the result for \( u_0 = \infty \).
The key idea is to consider the conjugate reciprocal polynomial, defined via the involution map \( ^\ast: \mathcal{P}_n \to \mathcal{P}_n \) given by  
\[
   P^\ast(u) := u^n \overline{P(1/\overline{u})}.
\]
This transformation interchanges the constant term and the leading coefficient of \( P \) and \( P^* \).  
For convenience, we introduce the notation \( u^\ast := 1/\overline{u} \) and observe that \( u^* = u \) for \( u \in \mathbb{T} \). Moreover, the mapping \( u \mapsto u^* \) defines a bijection of \( \Omega_\alpha \) onto itself. Applying this to the residual polynomials, we obtain  
\begin{equation}
\label{Rn star}
   R_{n}^{\Gamma_\alpha,w}(\cdot,\infty) = \bigl(R_{n}^{\Gamma_\alpha,w}(\cdot,0) \bigr)^\ast.
\end{equation}

The corresponding transformation of the Riemann maps follows from the fact that for each fixed \( v \in \Omega_\alpha \), there exists \( \gamma \in [0, 2\pi) \) such that  
\begin{equation}
\label{b*}
   \overline{b_{\Omega_\alpha}(u^*, v)} = e^{i\gamma} b_{\Omega_\alpha}(u, v^*).
\end{equation}
This relation arises because the left-hand side defines an analytic map from \( \Omega_\alpha \) onto \( \mathbb{D} \) that sends \( v^* \) to \( 0 \).  
In particular, choosing the normalization \( b_{\Omega_\alpha}(\infty, 0) > 0 \), we deduce that  
\begin{equation}
\label{b0}
   \overline{b_{\Omega_\alpha}(u^*, \infty)} = b_{\Omega_\alpha}(u, 0).
\end{equation}
Additionally, we note the identity  
   $b_{\Omega_\alpha}(u, 0) = u b_{\Omega_\alpha}(u, \infty)$,
which follows from the maximum principle by considering the ratio of these functions and noting that both are positive at \( u = \infty \). Furthermore, as explained just below \eqref{normal}, we have 
\begin{equation}
\label{ca}
   b_{\Omega_\alpha}(\infty, 0) = \lim_{u\to\infty} ub_{\Omega_\alpha}(u, \infty) =  
   \Cap(\Gamma_\alpha)=\sin(\alpha/2).
\end{equation}    
Returning to the asymptotic formula, we now obtain
\begin{multline}
	\qquad b_{\Omega_\alpha}(u,\infty)^nR_{n}^{\Gamma_\alpha,w}(u,\infty) 
	= b_{\Omega_\alpha}(u,\infty)^nu^n\overline{R_{n}^{\Gamma_\alpha,w}(u^\ast,0) }  \\
	= b_{\Omega_\alpha}(u,0)^n \overline{R_{n}^{\Gamma_\alpha,w}(u^\ast,0) } 
	= \overline{b_{\Omega_{\alpha}}(u^\ast,\infty)^nR_{n}^{\Gamma_\alpha,w}(u^\ast,0) }, \qquad
\end{multline}
which leads to the following result.
\begin{proposition}
Let $w$ be the weight function defined in \eqref{eq:weight_function_form}. Then \eqref{eq:residual_at_zero_limit} holds, and with the normalization $b_{\Omega_\alpha}(\infty,-1)>0$, we have
\begin{equation}
	b_{\Omega_{\alpha}}(u,\infty)^nR_{n}^{\Gamma_\alpha,w}(u,\infty)\longrightarrow 
	\frac{u+1-\sqrt{(u-e^{i\alpha})(u-e^{-i\alpha})}}{2(u+1)}
	\frac{b_{\Omega_\alpha}(u,-1)}{b_{\Omega_\alpha}(u,0)}F_w(u,\infty)
	\label{eq:residual_at_infinity_limit}
\end{equation}
uniformly on compact subsets of $\Omega_\alpha$ as $n\rightarrow \infty$.
\end{proposition}

\begin{proof}
Using \eqref{eq:mobius_transformation}, we observe that replacing $u$ with $u^*$ corresponds to replacing $z$ with $\overline{z}$. Recalling that $\overline{s(\overline{z})}=-s(z)$, we obtain the first factor on the right-hand side of \eqref{eq:residual_at_infinity_limit}. The remaining terms follow from \eqref{b*}--\eqref{b0}, together with the fact that $F_w(\infty, \infty)$ is positive.
\end{proof}

A key aspect of interest is determining the leading coefficient of $R_{n}^{\Gamma_\alpha,w}(\cdot,\infty)$, which equals $1/\|wT_n^{w}\|_{\Gamma_\alpha}$. Using \eqref{near inf}, we obtain
\[
   \frac{\Cap(\Gamma_\alpha)^n}{\|wT_n^w\|_{\Gamma_\alpha}} = 
   \lim_{u\rightarrow \infty}b_{\Omega_{\alpha}}(u,\infty)^nR_{n}^{\Gamma_\alpha,w}(u,\infty). 
\]
As a consequence of the previous proposition, we derive the following result.

\begin{corollary}
Let $w$ be the weight function defined in \eqref{eq:weight_function_form}. Then 
\begin{equation}
   \lim_{n\rightarrow \infty}\cW_n(\Gamma_\alpha,w) = 
   2\cos(\alpha/4)^2\exp\left(\int_{\Gamma_\alpha} \log |w(x)|\,d\mu_{\Gamma_\alpha}(x)\right).
\end{equation}
\end{corollary}

\begin{proof}
It suffices to compute the limit of the right-hand side of \eqref{eq:residual_at_infinity_limit} as $u\to\infty$. Since $s(z_\infty)=-1$, the first factor tends to $1$. 
To evaluate $b_{\Omega_\alpha}(\infty, -1)$, we note that a direct calculation using the Joukowski transformation gives
\[
   b_{\Omega_0}(z,0)  = \frac{-1+\sqrt{1-z^2}}{iz} = \frac{iz}{1+\sqrt{1-z^2}},
\]
where the square root is taken with a branch cut along $\sfA_0$ and chosen so that $\sqrt{1}=1$.
Substituting $z_\infty = -i\tan(\alpha/2)$, we obtain
\[
   b_{\Omega_\alpha}(\infty, -1) = b_{\Omega_0}(z_\infty, 0) 
   = \frac{\tan(\alpha/2)}{1+\sqrt{1+\tan^2(\alpha/2)}} 
   = \frac{\sin(\alpha/2)}{2\cos^2(\alpha/4)}.
\]
However, the numerator cancels out, as $b_{\Omega_\alpha}(\infty, 0) = \sin(\alpha/2)$ by \eqref{ca}.
Finally, from \eqref{hw}--\eqref{eq:outer_function}, we see that 
\[
   F_w(\infty,\infty)=\exp\left(-\int_{\Gamma_\alpha} \log |w(x)|\,d\mu_{\Gamma_\alpha}(x)\right),
\]
since the normalization ensures that $F_w(\infty,\infty)>0$. This completes the proof.
\end{proof}


\subsection{Residual polynomials relative to an arbitrary point}

To complete the picture, we determine the asymptotic behavior of the residual polynomial relative to an arbitrary point in $\overline{\C}\setminus\partial\D$. By symmetry, it suffices to consider points $u_0\in\D$. Indeed, for $|u_0|>1$, the relation $R_{n}^{\Gamma_\alpha,w}(\cdot,u_0)= R_{n}^{\Gamma_\alpha,w}(\cdot,u_0^\ast)^\ast$ implies the identity 
\begin{equation}b_{\Omega_\alpha}(u,\infty)^n R_{n}^{\Gamma_\alpha,w}(u,u_0) = e^{i\gamma}\overline{b_{\Omega_\alpha}(u^\ast,\infty)^nR_{n}^{\Gamma_\alpha,w}(u^\ast,u_0^\ast)}
	\label{eq:reflection_formula}
\end{equation}
for some $\gamma\in [0,2\pi)$.


\begin{proposition}
	Let $w$ be the weight function defined in \eqref{eq:weight_function_form}. For a fixed $u_0\in \D$, we define the parameters
\[
   \theta_0\in [0,2\pi), \;\; \tilde{\alpha}\in (0,\pi), \;\; z_{u_0}\in\{z : \Im z>0\}, 
   \; \mbox{ and } \; x_0\in (-1,1)
\] 	
    by the relations
	\begin{equation}
		\frac{e^{\pm i\alpha}-u_0}{1-e^{\pm i\alpha}u_0} = e^{i(\pm \tilde\alpha+\theta_0)},
		\label{eq:def_theta_0}
	\end{equation}
	\begin{equation}
		z_{u_0} := i\tan(\alpha/2)\frac{1+u_0}{1-u_0},
		\label{eq:z_u_0_def}
	\end{equation}
	and
	\begin{equation}
	x_0:=\tan(\alpha/2)\frac{\sin(\theta_0/2)-\Im(u_0e^{-i\theta_0/2})}{\cos(\theta_0/2)-\Re(u_0e^{-i\theta_0/2})}.
	\label{eq:x_0_def}
	\end{equation}
	Then, uniformly on compact subsets of $\Omega_0$, we have the asymptotic formula
	\begin{equation}
		\lim_{n\rightarrow \infty}b_{\Omega_\alpha}\bigl(u(z),\infty\bigr)^{n}R_{n}^{\Gamma_\alpha,w}\bigl(u(z),u_0\bigr) = \frac{1+s(z,z_{u_0})}{2}\frac{b_{\Omega_0}(z,x_0)}{b_{\Omega_0}(z,\overline{z}_{u_0})}F_w\bigl(u(z),u_0\bigr),
		\label{eq:asymptotics_z_domain}
	\end{equation}
	where
	\begin{equation} \label{su}
	s(z,z_{u_0}):=\frac{z_{u_0}-x_0}{\sqrt{z_{u_0}^2-1}}\frac{\sqrt{z^2-1}}{z-x_0},
	\; \mbox{ with } \; s(z_{u_0},z_{u_0})=1. \end{equation}
Here, we assume that $b_{\Omega_0}(z_{u_0}, x_0)$ and $b_{\Omega_0}(z_{u_0}, \overline{z}_{u_0})$ are positive, and $F_w$ is defined in \eqref{eq:outer_function}. 
	\label{prop:main_propsition}
\end{proposition}

\begin{proof}
Our strategy parallels the approach in Section \ref{subsec:res_origin}, where we related the residual polynomial at the origin to a weighted problem on $\Omega_0$. Here, we adapt this method for evaluation at an arbitrary point $u_0 \in \mathbb{D}$.
The key idea is to introduce an intermediate domain $\Omega_{\tilde{\alpha}}$ where the point $u_0$ is mapped to the origin. This allows us to utilize the known asymptotics for the origin case (evaluated at a shifted parameter) and then transfer the results back to our original setting.

First, consider the Möbius transformation
\[
    T_{u_0}(u) = \frac{u - u_0}{1 - \overline{u}_0 u},
\]
which maps $\mathbb{D}$ conformally onto itself and sends $u_0$ to $0$. Since $T_{u_0}$ preserves the unit circle, the image of the arc $\Gamma_\alpha$ under this map is another arc on the unit circle. Consequently, there exist parameters $\theta_0 \in [0, 2\pi)$ and $\tilde{\alpha} \in (0, \pi)$ satisfying \eqref{eq:def_theta_0} such that the rotated map $u \mapsto e^{-i\theta_0}T_{u_0}(u)$ maps $\Gamma_\alpha$ onto $\Gamma_{\tilde{\alpha}}$.
We define the conformal maps $\varphi_1: \Omega_\alpha \rightarrow \Omega_{\tilde{\alpha}}$ and $\varphi_2: \Omega_{\tilde{\alpha}} \rightarrow \Omega_0$ by
\begin{equation}
    \varphi_1(u) = e^{-i\theta_0} \frac{u - u_0}{1 - \overline{u}_0 u} \quad \mbox{ and } \quad
    \varphi_2(v) = i \tan(\tilde{\alpha}/{2})\frac{1 + v}{1 - v}.
\end{equation}
The mapping $\varphi_1$ is illustrated in Figure \ref{fig:conformal_maps_1}.

\begin{figure}[h]
\begin{center}
\begin{tikzpicture}[scale=1.7]

\begin{scope}
    \draw[dashed, very thick] (1,0) arc[start angle=0, end angle=360, radius=1];
    \draw[very thick] (1,0) arc[start angle=0, end angle=30, radius=1];
    \draw[very thick] (1,0) arc[start angle=0, end angle=-30, radius=1];

    \draw[->] (-1.5,0) -- (1.5,0) node[right] {};
    \draw[->] (0,-1.5) -- (0,1.5) node[above] {};
	\draw[thick,fill=white] (0.866,0.5) circle(0.05)node[above right] {$e^{i\alpha}$};
    \draw[thick,fill=white] (0.866,-0.5) circle(0.05)node[below right] {$e^{-i\alpha}$};
    \draw[thick,fill=black] (0.5,0.2) circle(0.03) node[above left] {$u_0$};

    \node[below right] at (1,0) {$1$};
	
    \node at (1,1.4) {$\Omega_\alpha$};
\end{scope}

\begin{scope}[xshift=5cm]
    \draw[dashed, very thick] (1,0) arc[start angle=0, end angle=360, radius=1];
    \draw[very thick] (1,0) arc[start angle=0, end angle=120, radius=1];
    \draw[very thick] (1,0) arc[start angle=0, end angle=-45, radius=1];

    \draw[->] (-1.5,0) -- (1.5,0) node[right] {};
    \draw[->] (0,-1.5) -- (0,1.5) node[above] {};
	\draw[thick,fill=white] (-0.5,0.866) circle(0.05)node[above left] {$e^{i(\tilde \alpha+\theta_0)}$};
    \draw[thick,fill=white] (0.707,-0.707) circle(0.05)node[below right] {$e^{i(-\tilde \alpha+\theta_0)}$};
    \draw[thick,fill=white] (0.7933,0.608) circle(0.05)node[above right] {$e^{i\theta_0}$};
    \draw[thick,fill=black] (0,0) circle(0.03) node[above left] {$0$};

    \node[below right] at (1,0) {$1$};
	
\end{scope}

\draw[->, thick] (1.5,0.5) to[out=30, in=150] (3.5,0.5);

\node[above] at (2.5,0.85) {$\frac{u-u_0}{1-\overline{u}_0 u}$};
\begin{scope}[xshift = 2.5cm, yshift=-3cm]
    \draw[dashed, very thick] (1,0) arc[start angle=0, end angle=360, radius=1];
    \draw[very thick] (1,0) arc[start angle=0, end angle=82.5, radius=1];
    \draw[very thick] (1,0) arc[start angle=0, end angle=-82.5, radius=1];

    \draw[->] (-1.5,0) -- (1.5,0) node[right] {};
    \draw[->] (0,-1.5) -- (0,1.5) node[above] {};
	\draw[thick,fill=white] (0.1305,0.9914) circle(0.05)node[above right] {$e^{i\tilde \alpha}$};
    \draw[thick,fill=white] (0.1305,-0.9914) circle(0.05)node[below right] {$e^{-i\tilde \alpha}$};
    \draw[thick,fill=black] (0,0) circle(0.03) node[above left] {$0$};

    \node[below right] at (1,0) {$1$};
	\node at (-1,-1.2) {$\Omega_{\tilde \alpha}$};

\end{scope}

\draw[->, thick] (5.5,-1.3) to[out=-120, in=20] (3.8,-2.5);

\node[above] at (4.8,-2.6) {$e^{-i\theta_0}$};

\draw[->, thick] (-0.5,-1.3) to[out=-60, in=160] (1.2,-2.5);
\node[above] at (0.2,-2.6) {$\varphi_1(u)$};

\end{tikzpicture}
\end{center}

\caption{The transformation of the domain $\Omega_\alpha$ (top left) to $\Omega_{\tilde{\alpha}}$ (bottom). The mapping $\varphi_1$ sends $u_0$ to the origin.}
\label{fig:conformal_maps_1}
\end{figure}
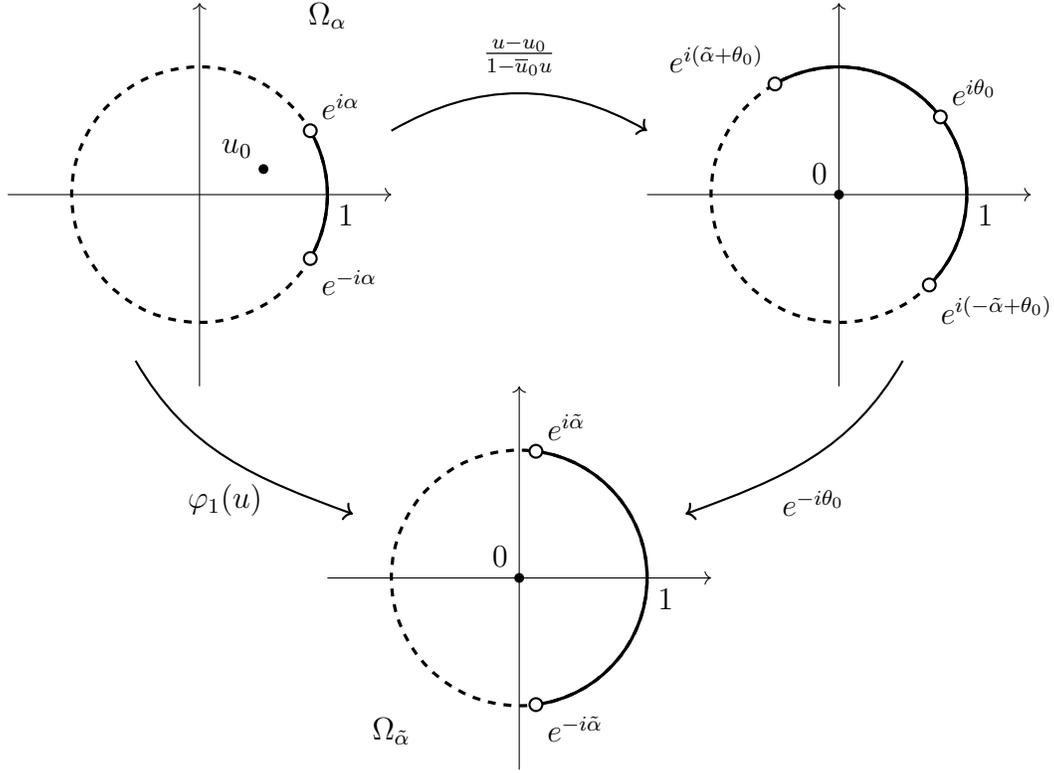
These mappings are Möbius transformations with inverses
\begin{equation}
    \varphi_1^{-1}(v) = \frac{e^{i \theta_0}v + u_0}{1 + \overline{u}_0 e^{i \theta_0} v} \quad \mbox { and } \quad
    \varphi_2^{-1}(\zeta) = \frac{\zeta - \zeta_0}{\zeta - \overline{\zeta}_0}, \quad \text{where } \; \zeta_0 = i \tan(\tilde{\alpha}/{2}).
\end{equation}
Let $\phi(u) := (\varphi_2 \circ \varphi_1)(u)$, so that $\phi(u_0) = \varphi_2(0) = \zeta_0$. Applying the same transport argument as in Lemma \ref{lem:identification}, we obtain the relation
\begin{equation}
    R_n^{\Gamma_\alpha,w}\bigl(\phi^{-1}(\zeta), \phi^{-1}(\zeta_0)\bigr) w\bigl(\phi^{-1}(\zeta)\bigr) = R_n^{\sfA_0,W_n}(\zeta, \zeta_0) W_n(\zeta),
    \label{eq:residual_relation_auxiliary_domain}
\end{equation}
where the weight $W_n$ on $\Omega_0$ is given by
\begin{equation}
    W_n(\zeta) = e^{i \theta} (\zeta - \zeta_\infty)^{k - n} \prod_{j=1}^k (\zeta - \overline{\zeta}_j)^{-1},
    \label{eq:transferred_weight_2}
\end{equation}
with $\zeta_j = \phi(u_j)$, $\zeta_\infty = \overline{\zeta}_0 = -i \tan(\tilde{\alpha}/{2})$, and $\theta\in[0, 2\pi)$ chosen such that $W_n(\zeta_0) > 0$.
The asymptotics of $R_n^{\sfA_0, W_n}(\cdot, \zeta_0) W_n(\cdot)$ is provided by Proposition \ref{prop:asymptotics_of_real_problem}. Substituting this into \eqref{eq:residual_relation_auxiliary_domain}, we find that uniformly on compact subsets of $\Omega_0$,
\begin{equation}
    \lim_{n\rightarrow \infty}b_{\Omega_\alpha}\bigl(\phi^{-1}(\zeta),\infty\bigr)^nR_n^{\Gamma_\alpha,w}\bigl(\phi^{-1}(\zeta),u_0\bigr) = \frac{1}{2} \biggl(1+\frac{\zeta_0 \sqrt{\zeta^2-1}}{\zeta \sqrt{\zeta_0^2-1}}\biggr)\frac{b_{\Omega_0}(\zeta,0)}{b_{\Omega_0}(\zeta,\zeta_\infty)}F_w\bigl(\phi^{-1}(\zeta),u_0\bigr).
    \label{eq:auxilliary_limit}
\end{equation}
Here, the normalization ensures $b_{\Omega_0}(\zeta_0,0)>0$ and $b_{\Omega_0}(\zeta_0,\zeta_\infty)>0$.

To express this result in terms of the variable $z \in \Omega_0$, we recall the conformal map $u: \Omega_0 \to \Omega_\alpha$ defined in \eqref{eq:mobius_transformation}. Its inverse is given by
\begin{equation}
    u^{-1}(\xi) = i\tan(\alpha/{2})\frac{1+\xi}{1-\xi}.
\end{equation}
We then consider the automorphism $\Phi : \Omega_0 \to \Omega_0$ given by the composition
\begin{equation}
    \Phi = u^{-1} \circ \varphi_1^{-1} \circ \varphi_2^{-1}.
\end{equation}
We observe that $\Phi$ maps the real line to itself. Moreover, by construction, we have
\[
    \Phi(\zeta_0) = u^{-1}\Bigl(\varphi_1^{-1}\bigl(\varphi_2^{-1}(\zeta_0)\bigr)\Bigr) = u^{-1}\bigl(\varphi_1^{-1}(0)\bigr) = u^{-1}(u_0) = z_{u_0}.
\]
To determine the image of the origin, we calculate
\begin{align*}
	\Phi(0) & = \bigl( u^{-1}\circ\varphi_1^{-1}\bigr)(-1) 
	 = u^{-1}\left(\frac{u_0-e^{i\theta_0}}{1-\overline{u}_0e^{i\theta_0}}\right) \\
	& = i\tan\left(\frac{\alpha}{2}\right)\frac{1-\overline{u}_0e^{i\theta_0}+u_0-e^{i\theta_0}}{1-\overline{u}_0e^{i\theta_0}-u_0+e^{i\theta_0}}.
\end{align*}
Using Euler's formula and the definition of $x_0$ in \eqref{eq:x_0_def}, this simplifies to $\Phi(0) = x_0$. 
Since $\Phi$ is an automorphism of $\Omega_0$ mapping the origin to $x_0 \in (-1,1)$, it must be of the form
\begin{equation}
    \Phi(\zeta) = \frac{\zeta + x_0}{1 + x_0 \zeta}, \quad \mbox{ with inverse } \; \Phi^{-1}(z) = \frac{z - x_0}{1 - x_0 z}.
    \label{eq:Phi_def}
\end{equation}

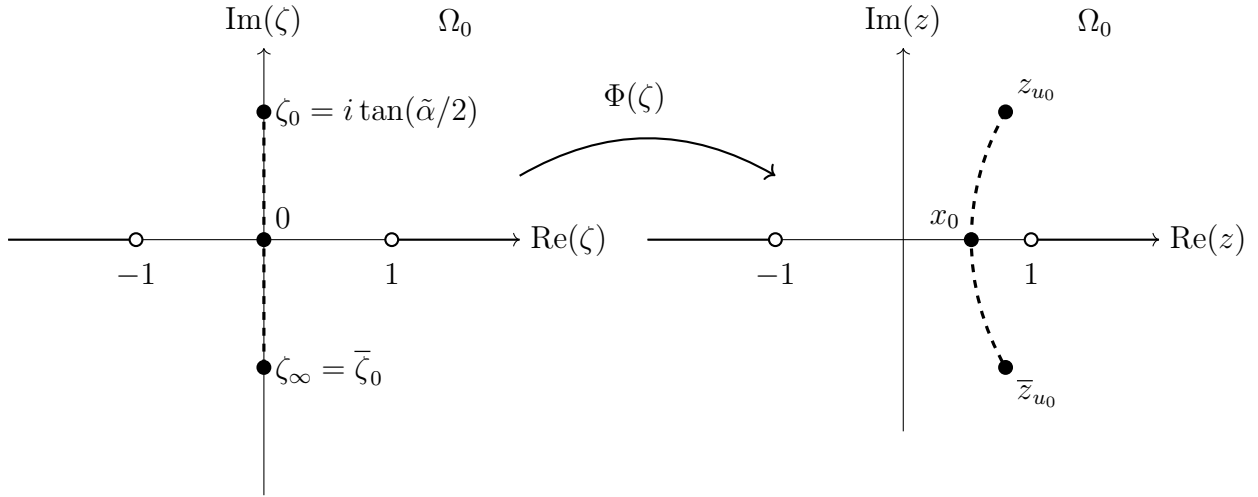
\begin{figure}[h]
\begin{center}
\begin{tikzpicture}[scale=1.7]

\begin{scope}
    \draw[->] (-2,0) -- (2,0) node[right] {$\text{Re}(\zeta)$};
    \draw[->] (0,-2) -- (0,1.5) node[above] {$\text{Im}(\zeta)$};
    
    \draw[thick] (-2,0) -- (-1,0);
    \draw[thick] (1,0) -- (2,0);

    \draw[thick,fill=white] (-1,0) circle(0.05);
    \draw[thick,fill=white] (1,0) circle(0.05);

    \node[below] at (-1,-0.1) {$-1$};
    \node[below] at (1,-0.1) {$1$};
    \draw[thick,fill=black] (0,1) circle(0.05) node[right] {$\zeta_0 = i\tan(\tilde\alpha/{2})$};
    \draw[thick,fill=black] (0,-1) circle(0.05) node[right] {$\zeta_\infty=\overline\zeta_0$};
    \draw[thick,fill=black] (0,0) circle(0.05) node[above right] {$0$};
    	\draw[dashed, very thick] (0,-1) -- (0,1);
    \node at (1.5,1.7) {$\Omega_0$};
\end{scope}

\begin{scope}[xshift=5cm]
    \draw[->] (-2,0) -- (2,0) node[right] {$\text{Re}(z)$};
    \draw[->] (0,-1.5) -- (0,1.5) node[above] {$\text{Im}(z)$};

	\draw[dashed, very thick] (0.8,1) arc[start angle=150, end angle=210, radius=2];

	\draw[thick,fill=black] (0.5320508076,0) circle(0.05)node[above left] {$x_0$};
	\draw[thick,fill=black] (0.8,1) circle(0.05)node[above right] {$z_{u_0}$};
    \draw[thick,fill=black] (0.8,-1) circle(0.05)node[below right] {$\overline{z}_{u_0}$};
    
    \draw[thick] (-2,0) -- (-1,0);
    \draw[thick] (1,0) -- (2,0);
	
    \draw[thick,fill=white] (-1,0) circle(0.05);
    \draw[thick,fill=white] (1,0) circle(0.05);

    \node[below] at (-1,-0.1) {$-1$};
    \node[below] at (1,-0.1) {$1$};
    \node at (1.5,1.7) {$\Omega_0$};
    
\end{scope}

\draw[->, thick] (2,0.5) to[out=30, in=150] (4,0.5);

\node[above] at (2.9,0.9) {$\Phi(\zeta)$};
\end{tikzpicture}
\end{center}

\caption{The Möbius transformation $\Phi$ mapping the auxiliary parameter plane (left) to the original parameter plane (right). The map is determined by the correspondence of the points $0 \mapsto x_0$ and $\zeta_0 \mapsto z_{u_0}$.}
\label{fig:conformal_maps_2}
\end{figure}

Substituting $\zeta = \Phi^{-1}(z)$ into the terms of \eqref{eq:auxilliary_limit}, we compute
\begin{equation}
    \frac{\zeta_0 \sqrt{\zeta^2-1}}{\zeta \sqrt{\zeta_0^2-1}}
    = \frac{\Phi^{-1}(z_{u_0})\sqrt{\Phi^{-1}(z)^2-1}}{\Phi^{-1}(z)\sqrt{\Phi^{-1}(z_{u_0})^2-1}}
    = \frac{z_{u_0}-x_0}{\sqrt{z_{u_0}^2-1}}\frac{\sqrt{z^2-1}}{z-x_0}
    = s(z, z_{u_0}).
\end{equation}
Similarly, the ratio of Blaschke factors transforms as
\begin{equation}
    \frac{b_{\Omega_0}(\zeta,0)}{b_{\Omega_0}(\zeta,\zeta_\infty)}
    = \frac{b_{\Omega_0}\bigl(\Phi^{-1}(z), \Phi^{-1}(x_0)\bigr)}{b_{\Omega_0}\bigl(\Phi^{-1}(z),\Phi^{-1}(\overline {z}_{u_0})\bigr)}
    = \frac{b_{\Omega_0}(z,x_0)}{b_{\Omega_0}(z,\overline{z}_{u_0})}.
\end{equation}
Finally, noting that $\phi^{-1}(\zeta) = u(z)$, substituting these identities into \eqref{eq:auxilliary_limit} yields the desired asymptotic formula \eqref{eq:asymptotics_z_domain}.
\end{proof}

We summarize our findings for the domain $\Omega_\alpha$, presenting the results in a form consistent with \cite{eichinger17}. Recall the mapping \(\lambda : \Omega_\alpha \to \mathsf{S}_{\pi/4}\) defined in \eqref{la} by 
\begin{equation}
	\lambda(u) := \left(\frac{ue^{i\alpha}-1}{u-e^{i\alpha}}\right)^{1/4}.
	\label{eq:lambda}
\end{equation}

\begin{theorem}
	\label{thm:weighted_residual_asymptotics}
	Let $w$ be the weight function defined in \eqref{eq:weight_function_form}. For a fixed $u_0\in \D$, we have 
	\begin{multline}
	 \lim_{n\rightarrow \infty}b_{\Omega_\alpha}(u,\infty)^nR_n^{\Gamma_\alpha,w}(u,u_0) \\ = \frac{1}{2}\left(1+\frac{h(\lambda(u),\lambda(u_0))}{h(\lambda(u_0),\lambda(u_0))}\right) \frac{|\lambda(u_0)|^2 \bigl(\lambda(u)^2-|\lambda(u_0)|^2\bigr) \bigl(\lambda(u)^2+\lambda(u_0)^2\bigr)}{\lambda_0^2 \, \bigl(\lambda(u)^2+|\lambda(u_0)|^2\bigr) \bigl(\lambda(u)^2-\overbar{\lambda(u_0)^2}\bigr)} F_w(u,u_0)  \quad
		\label{eq:local_uniform_asymptotics}
	\end{multline}
	uniformly on compact subsets of $\Omega_\alpha$, where the auxiliary function $h$ is given by
    \begin{equation} \label{h}
    	h(\lambda,\lambda_0) = \frac{\lambda^2}{(\lambda^2-|\lambda_0|^2)(\lambda^2+|\lambda_0|^2)}.
    \end{equation}
\end{theorem}

\begin{proof}
The result follows by transferring \eqref{eq:asymptotics_z_domain} into the domain $\mathsf{S}_{\pi/4}$ of the variable $\lambda$. To this end, consider the Möbius transformation $g:\overline{\mathbb{C}} \to \overline{\mathbb{C}}$ defined by
\[
   g(z) = \frac{z-1}{z+1}, \quad \mbox{with inverse } \;
   g^{-1}(\xi) = \frac{1+\xi}{1-\xi}.
\]
Note that $g$ maps the domain $\Omega_0$ onto $\overline{\mathbb{C}} \setminus [0,\infty)$.
Additionally, let $\Phi$ be the automorphism of $\Omega_0$ defined in \eqref{eq:Phi_def}. Since $\Phi$ preserves the real line and $\Phi(\pm 1) = \pm 1$, the composition $g \circ \Phi \circ g^{-1}$ is a Möbius transformation that preserves the ray $[0,\infty)$ and fixes both $0$ and $\infty$. Consequently, it must be a scaling $\xi \mapsto r \xi$ for some $r>0$. Observing that $g(i\mathbb{R}) = \mathbb{T}$, we conclude that $(g \circ \Phi)(i\mathbb{R}) = r\mathbb{T}$. Since the parameters $z_{u_0}$, $x_0$, and $\overline{z}_{u_0}$ all lie on $\Phi(i\mathbb{R})$, their images under $g$ lie on the circle of radius $r$.

We define the function $v(z)$ by choosing a branch of $\sqrt{g(z)}$ that maps $\Omega_0$ conformally onto the upper half-plane. Consistent with our earlier choice of branch cuts along $\mathbb{R} \setminus (-1,1)$, we have the identities $\sqrt{\overline{z}-1} = -\overline{\sqrt{z-1}}$ and $\sqrt{\overline{z}+1} = \overline{\sqrt{z+1}}$. This implies the symmetry relation
\[
   v(\overline{z}) = -\overline{v(z)}.
\]
Setting $v_0 = v(z_{u_0})$, we thus have $v(\overline{z}_{u_0}) = -\overline{v}_0$. Furthermore, since $|g(x_0)| = |g(z_{u_0})|$ and $g(x_0)<0$, it follows that $v(x_0) = i|v_0|$. The mapping properties of $v$ are illustrated in Figure \ref{fig:v_map}.

\begin{figure}[h]
\begin{center}
\begin{tikzpicture}[scale=1.7]

\begin{scope}
    \draw[->] (-2,0) -- (2,0) node[right] {$\text{Re}(z)$};
    \draw[->] (0,-1.5) -- (0,1.5) node[above] {$\text{Im}(z)$};

	\draw[dashed, very thick] (0.8,1) arc[start angle=150, end angle=210, radius=2];

	\draw[thick,fill=black] (0.5320508076,0) circle(0.05)node[above left] {$x_0$};
	\draw[thick,fill=black] (0.8,1) circle(0.05)node[above right] {$z_{u_0}$};
    \draw[thick,fill=black] (0.8,-1) circle(0.05)node[below right] {$\overline{z}_{u_0}$};
    
    \draw[thick] (-2,0) -- (-1,0);
    \draw[thick] (1,0) -- (2,0);
	
    \draw[thick,fill=white] (-1,0) circle(0.05);
    \draw[thick,fill=white] (1,0) circle(0.05);

    \node[below] at (-1,-0.1) {$-1$};
    \node[below] at (1,-0.1) {$1$};
    \node at (1.5,1.7) {$\Omega_0$};
    
\end{scope}

\draw[->, thick] (2.2,0.5) to[out=30, in=150] (3.8,0.5);

\node[above] at (3.0,0.8) {$v(z)$};

\begin{scope}[xshift=6cm]
    
    \fill[gray!15] (-1.9,-1) -- (-1.9,1.4) -- (1.9,1.4) -- (1.9,-1) -- cycle;
    \draw[->] (-2,-1) -- (2,-1) node[right] {$\text{Re}(v)$};
    \draw[->] (0,-1.5) -- (0,1.5) node[above] {$\text{Im}(v)$};

	\draw[dashed, very thick] (1.5*0.866,-1.5*0.5+0.5) arc[start angle=30, end angle=150, radius=1.5];
	\draw[dashed, ] (1.5,-1) arc[start angle=0, end angle=30, radius=1.5];
	\draw[dashed, ] (-1.5,-1) arc[start angle=180, end angle=150, radius=1.5];

	\draw[thick,fill=black] (0,0.5) circle(0.05)node[above left] {$i|v_0|$};
	\draw[thick,fill=black] (1.5*0.866,-1.5*0.5+0.5) circle(0.05)node[above right] {$v_0$};
    \draw[thick,fill=black] (-1.5*0.866,-1.5*0.5+0.5) circle(0.05)node[above left] {$-\overline{v}_0$};
      
     \draw[thick,fill=white] (0,-1) circle(0.05);
	 \node[below] at (0.2,-1.1) {$0$};

\end{scope}
\end{tikzpicture}
\end{center}

\caption{The conformal map $v(z) = \sqrt{(z-1)/(z+1)}$ transforms the slit domain $\Omega_0$ (left) onto the upper half-plane (right). The points $z_{u_0}$, $\overline{z}_{u_0}$, and $x_0$ are mapped to $v_0$, $-\overline{v}_0$, and $i|v_0|$ respectively.}
\label{fig:v_map}
\end{figure}
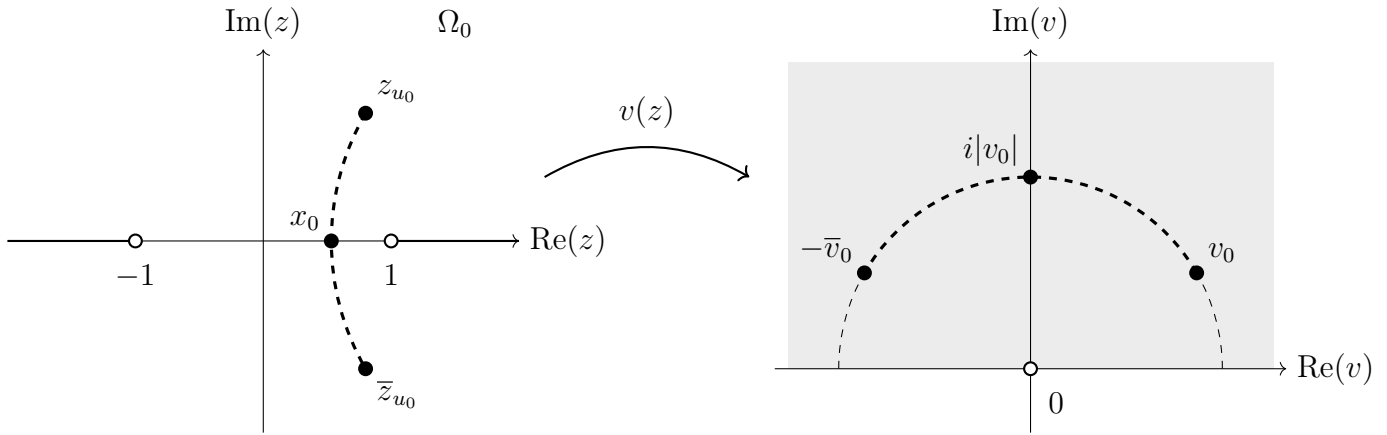

Recall the mapping $\lambda(u)$ defined in \eqref{eq:lambda}. With $u(z)$ from \eqref{eq:mobius_transformation}, a direct computation shows that
\begin{equation}
   i\lambda\bigl( u(z) \bigr)^2 = v(z).
\end{equation}
With $\lambda = \lambda(u(z))$ and $\lambda_0 = \lambda(u_0)$, this identity implies $v(z) = i\lambda^2$ and $v_0 = i\lambda_0^2$.
We now rewrite the terms in the asymptotic formula \eqref{eq:asymptotics_z_domain} using these variables. Using the relation $z = g^{-1}(v^2)$ and recalling the definition of the $h$ in \eqref{h}, the term $s(z, z_{u_0})$ transforms as follows:
\begin{equation}
   s(z, z_{u_0}) = \frac{(v_0^2+|v_0|^2) v(z)}{v_0 (v(z)^2+|v_0|^2)} 
   = \frac{\lambda^2 (\lambda_0^2 - |\lambda_0|^2)(\lambda_0^2 + |\lambda_0|^2)}{\lambda_0^2 (\lambda^2 - |\lambda_0|^2)(\lambda^2 + |\lambda_0|^2)}
   = \frac{h(\lambda, \lambda_0)}{h(\lambda_0, \lambda_0)}.
\end{equation}
Similarly, the ratio of Blaschke factors becomes
\begin{equation}
	\frac{b_{\Omega_0}(z,x_0)}{b_{\Omega_0}(z,\overline{z}_{u_0})} 
    = i \frac{v(z)-i|v_0|}{v(z)+i|v_0|} \cdot \frac{|v_0|}{v_0}\frac{v(z)+v_0}{v(z)+\overline{v}_0} 
    = \frac{|\lambda_0|^2 (\lambda^2-|\lambda_0|^2) (\lambda^2+\lambda_0^2)}{\lambda_0^2 (\lambda^2+|\lambda_0|^2) (\lambda^2-\overbar{\lambda_0^2})}.
\end{equation}
Substituting these identities into \eqref{eq:asymptotics_z_domain} yields the formula \eqref{eq:local_uniform_asymptotics}, completing the proof.

\end{proof}

A central quantity of interest is the asymptotic extremal value at a given point $u_0$. To characterize this behavior, we set $\lambda = \lambda(u)$ and $\lambda_0 = \lambda(u_0)$, and introduce the kernel function $k_{\Omega_\alpha}$ defined by
\begin{equation}
	k_{\Omega_\alpha}(u,u_0) := k_{\C_+}\bigl(\lambda,\lambda_0\bigr)=\frac{2\lambda \overline{\lambda}_0 }{(\lambda+\overline{\lambda}_0)^2}.
\end{equation}
Recall that for the unweighted case ($w \equiv 1$), \cite{eichinger17} established that
\[\lim_{n\rightarrow \infty}e^{ng_{\Omega_\alpha}(u_0,\infty)}\|T_n^{\Gamma_\alpha}(\cdot,u_0)\| = \frac{1}{k_{\Omega_\alpha}(u_0,u_0)}.\]
We now extend this result to weighted residual polynomials with rational weights, paving the way for the broader class of admissible weight functions considered in the next section.

\begin{theorem}
	\label{thm:weighted_residual_asymptotics_at_point}
	Let $w$ be the weight function defined in \eqref{eq:weight_function_form}. Then for any point $u_0\in \Omega_\alpha$, we have 	
	\begin{equation}
		\lim_{n\rightarrow \infty}e^{ng_{\Omega_\alpha}(u_0,\infty)}\|wT_n^{\Gamma_{\alpha},w}(\cdot ,u_0)\|_{\Gamma_\alpha} = k_{\Omega_\alpha}(u_0,u_0)^{-1}\exp\left(\int_{\Gamma_\alpha}\log w(x)\,\omega(u_0,dx,\Omega_{\alpha})\right).
		\label{eq:residual_max_asymptotics}
	\end{equation}
\end{theorem}

\begin{proof}
First, assume that $u_0\in \D$. By the normalization in \eqref{normal}, we have $b_{\Omega_\alpha}(u_0,\infty) = e^{-g_{\Omega_\alpha}(u_0,\infty)}$. Hence, it follows from \eqref{eq:local_uniform_asymptotics} that as $n\rightarrow \infty$,
	\[
		e^{-n g_{\Omega_\alpha}(u_0,\infty)} R_n^{\Gamma_{\alpha},w}(u_0,u_0) \longrightarrow \frac{2|\lambda_0|^2}{(\lambda_0+\overline{\lambda}_0)^2} F_w(u_0,u_0) = k_{\Omega_\alpha}(u_0,u_0)F_w(u_0,u_0).
    \]
	Recalling that
	\[F_w(u_0,u_0) = \exp\left(-\int_{\Gamma_\alpha} \log w(x)\,\omega(u_0,dx,\Omega_{\alpha})\right),\]
	we obtain \eqref{eq:residual_max_asymptotics} as a direct consequence of the duality relation \eqref{eq:dual_relation}.

The corresponding result for the exterior $\{u: |u|>1\}$ follows from symmetry considerations. Let $u^\ast = 1/\overline{u}$. For $u_0\in \D$, the reflection formula \eqref{eq:reflection_formula} implies that
	\[
	e^{-ng_{\Omega_\alpha}(u_0^\ast,\infty)}R_{n}^{\Gamma_\alpha,w}(u_0^\ast,u_0^\ast) 
	= e^{-ng_{\Omega_\alpha}(u_0,\infty)}{R_{n}^{\Gamma_\alpha,w}(u_0,u_0)} 
	\longrightarrow k_{\Omega_\alpha}(u_0,u_0){F_w(u_0,u_0)}
	\]
	as $n\rightarrow \infty$. Since $\lambda(u^\ast) = \overline{\lambda(u)}$, we have $k_{\Omega_\alpha}(u_0^\ast,u_0^\ast) = k_{\Omega_\alpha}(u_0,u_0)$. Furthermore, since the weight $w$ is real-valued on $\Gamma_\alpha$, the Schwarz reflection principle implies that $F_w(u^\ast,u^\ast) = {F_w(u,u)}$. Thus, the limit holds for any $u_0^\ast$ in the exterior of the unit disk.

Finally, for $u_0\in \T\setminus \Gamma_\alpha$, we apply \cite[Lemma 2.7]{eichinger17}, suitably adjusted to the weighted setting. This ensures that the function
    \[
       u_0\longmapsto \lim_{n\rightarrow \infty} \Bigl( e^{ng_{\Omega_\alpha}(u_0,\infty)} \|wT_n^{\Gamma_\alpha, w}(\cdot,u_0)\|_{\Gamma_\alpha} \Bigr)
    \]
    is continuous on $\Omega_\alpha$, which completes the proof.	

\end{proof}
\begin{remark}
This result can naturally be interpreted through the theory of Hardy spaces. To describe how this is done, we let \(\Gamma\) denote a Jordan arc of class \(C^{1+\alpha}\) with complement \(\Omega\). If \(\psi\) denotes the conformal map from the exterior of the unit disk to \(\Omega\) mapping \(\infty\) to \(\infty\) then we let \(H^2(\Omega)\) consist of those functions \(F\) such that
\[F\circ \psi\in H^2(\bbD).\]
Such an \(F\) has non-tangential boundary values from either side (a.e. with respect to arc-length) on \(\Gamma\). We label these two boundary values by \(F_+\) and \(F_-\). Given a point \(z_0\in \Omega\), a norm on \(H^2(\Omega)\) is given by
\[\|F\|_{\omega(z_0)}^2= \int_{\Gamma}\bigl(|F_+(x)|^2+|F_-(x)|^2\bigr) \omega(z_0,dx,\Omega).\]
With this norm \(H^2(\Omega)\) becomes a reproducing kernel Hilbert space whose associated kernel is denoted by \(k_{\omega(z_0)}(z,\zeta)\). Classical Hilbert space theory shows that
\[\frac{1}{k_{\omega(z_0)}(z_0,z_0)} = \inf\left\{\int_{\Gamma}(|F_+(x)|^2+|F_-(x)|^2)\omega(z_0,dx,\Omega): F\in H^2(\Omega),\, F(z_0) = 1\right\}.\]
This quantity appears in relation to weighted residual polynomials through \cite[Conjecture 1.7]{buchecker-eichinger-zinchenko25} which hypothezises that if \(\Gamma\) is of class \(C^{1+\alpha}\) and \(w\) is an upper-semicontinuous weight function then
\[\lim_{n\rightarrow \infty}e^{ng_{\Omega}(z_0,\infty)}\|wT_n^{\Gamma,w}(\cdot,z_0)\|_{\Gamma} = \frac{1}{k_{\omega(z_0)}(z_0,z_0)}\exp\left(\int_{\Gamma}\log w(x)\omega(z_0,dx,\Omega_\alpha)\right).\]
In our current setting, where we investigate circular arcs, \cite[Theorem 5.4.10]{buchecker25} establishes the relation 
\[k_{\omega(u_0)}(u_0,u_0) = k_{\Omega_\alpha}(u_0,u_0).\]
Hence, Theorem \ref{thm:weighted_residual_asymptotics_at_point} proves the conjecture in the case of circular arcs and rational weights. Our aim is now to lift this result to a larger class of weight functions.
\end{remark}

\section{General Weight functions}
\label{sec:general_weight}
We now relax the assumptions on the weight function to extend formulas such as \eqref{eq:chebyshev_norm_asymptotics} to a broader class of weights. For clarity of presentation, we fix $u_0\in \Omega_\alpha=\overline{\C}\setminus \Gamma_\alpha$ and shift our focus from \(R_n^{\Gamma_\alpha,w}(\cdot,u_0)\) to \(T_n^{\Gamma_\alpha,w}(\cdot,u_0)=:T_n^{w}(\cdot,u_0)\). It should be stressed that, by \eqref{eq:chebyshev_residual_relation}, these are essentially the same object. 

To begin, let $w:\Gamma_\alpha\rightarrow (0,\infty)$ be a positive continuous function on $\Gamma_\alpha$. Suppose there exist polynomials $Q_1$ and $Q_2$, both nonvanishing on $\T$, such that
\begin{equation}
	\frac{1}{|Q_1(u)|}\leq w(u)\leq \frac{1}{|Q_2(u)|}
	\label{eq:polynomial_weight_inequalities}
\end{equation}
for all $u\in \Gamma_\alpha$. Then it is easily seen that
\begin{align} \notag
	\|T_n^{1/Q_1}(\cdot,u_0)/Q_1\|_{\Gamma_\alpha}&\leq \|T_n^{w}(\cdot,u_0)/Q_1\|_{\Gamma_\alpha}\leq \|w T_n^{w}(\cdot,u_0)\|_{\Gamma_\alpha}\\&\leq\|T_n^{1/Q_2}(\cdot,u_0)w\|_{\Gamma_\alpha}\leq \|T_n^{1/Q_2}(\cdot,u_0)/Q_2\|_{\Gamma_\alpha}.	
\end{align}
Without loss of generality, we may assume that the polynomials $Q_1$ and $Q_2$ have all their zeros in the set $\{u: |u|> 1\}$. To see that this assumption is non-restrictive, just observe that
\[
   |u-u_j| = |1-u\overline{u}_j| 
\]
for $u\in \T$. Consequently, if $P$ is a polynomial with a zero of order $n$ at $u_0\in \D$, the polynomial 
\[
   \frac{P(u)(1-u\overline{u}_0)^n}{(u-u_0)^n}
\] 
has the same absolute value on $\T$ but now has a zero at $1/\overline{u}_0$, which lies in the set $\{u: |u|>1\}$. 

Now, let $\epsilon>0$ be given. From Theorem \ref{thm:weighted_residual_asymptotics_at_point}, we have
\[\|T_n^{1/Q_k}(\cdot,u_0)/Q_k\|_{\Gamma_\alpha}\sim e^{-ng_{\Omega_\alpha}(u_0,\infty)}k_{\Omega_\alpha}(u_0,u_0)^{-1}\exp\left(-\int_{\Gamma_\alpha}\log Q_k(x)\,\omega(u_0,dx,\Omega_{\alpha})\right)\]
for $k\in \{1,2\}$.
By taking $n$ large enough, we can ensure that
\begin{multline*}
\qquad k_{\Omega_\alpha}(u_0,u_0)^{-1} \exp\left(-\int_{\Gamma_\alpha}\log Q_1(x)\,
\omega(u_0,dx,\Omega_{\alpha})\right)-\epsilon\leq 
\|w T_n^{w}(\cdot,u_0)\|_{\Gamma_\alpha}e^{ng_{\Omega_\alpha}(u_0,\infty)} \\
\leq k_{\Omega_\alpha}(u_0,u_0)^{-1}\exp\left(-\int_{\Gamma_\alpha}\log Q_2(x)\,\omega(u_0,dx,\Omega_{\alpha})\right)+\epsilon. \qquad	
\end{multline*}
Since $\epsilon>0$ was arbitrary, it follows that
\begin{multline}
k_{\Omega_\alpha}(u_0,u_0)^{-1}\exp\left(-\int_{\Gamma_\alpha}\log Q_1(x)\,
\omega(u_0,dx,\Omega_{\alpha})\right)
\leq \liminf_{n\rightarrow \infty} 
\|w T_n^{w}(\cdot,u_0)\|_{\Gamma_\alpha}e^{ng_{\Omega_\alpha}(u_0,\infty)}\label{eq:polynomial_approximation_above_below} \\
\leq \limsup_{n\rightarrow \infty} 
\|wT_n^{w}(\cdot,u_0)\|_{\Gamma_\alpha}e^{ng_{\Omega_\alpha}(u_0,\infty)} 
\leq k_{\Omega_\alpha}(u_0,u_0)^{-1}\exp\left(-\int_{\Gamma_\alpha}\log Q_2(x)\,
\omega(u_0,dx,\Omega_{\alpha})\right). 
\end{multline}
This holds for any choice of polynomials $Q_1$ and $Q_2$ satisfying \eqref{eq:polynomial_weight_inequalities}. Our next step is to construct such polynomials so that the inequalities in \eqref{eq:polynomial_weight_inequalities} become approximate equalities.

Let $\epsilon>0$ be chosen such that $0<2\epsilon<\min_{u\in \Gamma_\alpha} w(u)$. By the Stone--Weierstra{\ss} theorem, we can approximate the function $1/w$ using polynomials in $u$ and $\overline{u}$. 
Specifically, there exist polynomials $P_1$ and $P_2$ satisfying
\begin{equation}
	w(u)-2\epsilon < 1/|P_{1}(u,\overline{u})|< w(u)-\epsilon\leq w(u)+\epsilon < 1/|P_2(u,\overline{u})|< w(u)+2\epsilon
	\label{eq:stone_weierstrass_generated_inequality}
\end{equation}
for all $u\in \Gamma_\alpha$. 
Since $\overline{u} = u^{-1}$ for $u\in \T$, the polynomials $P_k$ (for $k=1, 2$) can be rewritten as 
\[P_k(u,\overline{u}) = \sum_{l,j=1}^{m}a_{l,j}^{(k)}u^l\overline{u}^j = \sum_{l,j = 1}^{m}a_{l,j}^{(k)}u^lu^{-j} = u^{-m}\sum_{l,j=1}^{m}a_{l,j}^{(k)}u^{l}u^{m-j}\]
for $u\in \T$.
Thus, defining
\[
	Q_{k}(u) := \sum_{l,j=1}^{m}a_{l,j}^{(k)}u^{l}u^{m-j},
\]
we obtain polynomials $Q_k$ whose absolute values on $\T$ coincide with those of $P_k$. 
Moreover, we may assume that $Q_k$ is zero-free on $\T$. For if a zero appears, it can be perturbed by slightly dilating it with a factor $1+\delta$ (for some $\delta>0$), which continuously changes the modulus of the polynomial. If this perturbation is sufficiently small, the inequalities in \eqref{eq:stone_weierstrass_generated_inequality} remain valid, ensuring that
\begin{equation}
	w(u)-2\epsilon < 1/|Q_1(u)|< w(u)-\epsilon\leq w(u)+\epsilon < 1/|Q_2(u)|< w(u)+2\epsilon.
\end{equation}
Since this holds for all sufficiently small $\epsilon$, we can make the two integral in \eqref{eq:polynomial_approximation_above_below} arbitrarily close. 
This leads to the asymptotic formula
\begin{equation}
	\|w T_n^{w}(\cdot,u_0)\|_{\Gamma_\alpha}\sim  e^{-ng_{\Omega_\alpha}(u_0,\infty)}k_{\Omega_\alpha}(u_0,u_0)^{-1}\exp\left(\int_{\Gamma_\alpha} \log w(x)\,\omega(u_0,dx,\Omega_{\alpha})\right).
	\label{eq:limit_chebyshev_norm}
\end{equation}

By employing a similar approach, now approximating step functions from above and below, we can show that \eqref{eq:limit_chebyshev_norm} holds for any Riemann integrable weight function $w$ on $\Gamma_\alpha$ that satisfies the additional condition that there exists a constant $M\geq 1$ such that
\begin{equation}
	1/M\leq w(u)\leq M
	\label{eq:weight_bounds}
\end{equation}
for $u\in \Gamma_\alpha$.
Since the necessary arguments closely follow those in \cite[Chapter I.7]{bernstein30} (see also \cite[Appendix A.2]{christiansen-eichinger-rubin23} for a summary), we omit them here.

\subsection{Allowing for zeros}
The most general form of our asymptotic formula in the context of circular arcs is given as follows.

\begin{theoremm}
	Let $w:\Gamma_\alpha\rightarrow [0,\infty)$ be a weight function of the form
	\begin{equation}
	w(u) = w_0(u)\prod_{j=1}^{k}|u-u_j|^{s_j},
	\end{equation}
	where $w_0$ is a Riemann integrable function satisfying \eqref{eq:weight_bounds}, and where $u_j\in \Gamma_\alpha$ and $s_j\in \R$. Then, for any $u_0\in \Omega_\alpha$,
	\begin{equation}
		\|w T_n^{w}(\cdot,u_0)\|_{\Gamma_\alpha}\sim  e^{-ng_{\Omega_\alpha}(u_0,\infty)}k_{\Omega_\alpha}(u_0,u_0)^{-1}\exp\left(\int_{\Gamma_\alpha} \log w(x)\,\omega(u_0,dx,\Omega_{\alpha})\right).
	\end{equation}
	In particular,
	\begin{equation}
		\lim_{n\rightarrow \infty}\cW_n(\Gamma_\alpha,w)= 2\cos(\alpha/4)^2\exp\left(\int_{\Gamma_\alpha} \log w(x)\,d\mu_{\Gamma_\alpha}(x)\right).
		\label{eq:chebyshev_norm_general_formula}
	\end{equation}
	\label{thm:main_norm_asymptotics}
\end{theoremm}

\begin{proof}
	Similar to the approach in \cite{bernstein31}, we show how an additional factor $|u-u_j|^{s_j}$ can be incorporated into \eqref{eq:chebyshev_norm_general_formula}, under the assumption that \eqref{eq:chebyshev_norm_general_formula} holds for $w = w_0$. The remaining result can then be proved by induction. 
	
Let us begin by assuming that 
	\begin{equation}
	w(u) = w_0(u)|u-u_j|,
	\end{equation}
	where $w_0$ is a weight for which \eqref{eq:chebyshev_norm_general_formula} holds. 
	By incorporating the zero of the weight $w$ into the polynomial, we find that if $u_0\neq \infty$, then
	\begin{multline*}
		\qquad \frac{1}{|u_0-u_j|}\|w T_n^{w}(\cdot,u_0)\|_{\Gamma_\alpha}\geq \|w_0 T_{n+1}^{w_0}(\cdot,u_0)\| \\ \sim e^{-(n+1)g_{\Omega_\alpha}(u_0,\infty)}k_{\Omega_\alpha}(u_0,u_0)^{-1}\exp\left(\int_{\Gamma_\alpha} \log w_0(x)\,\omega(u_0,dx,\Omega_{\alpha})\right). \qquad
	\end{multline*}
	At the same time, it is clear that
	\[
		\|w T_n^{w}(\cdot,\infty)\|_{\Gamma_\alpha}\geq \|w_0 T_{n+1}^{w_0}(\cdot,\infty)\| \sim 2\cos(\alpha/4)^2\Cap(\Gamma_\alpha)^{n+1}\exp\left(\int_{\Gamma_\alpha} \log w_0(x)\,d\mu_{\Gamma_\alpha}(x)\right).
	\]
	The logarithmic integral of $w$ can be split into two parts:
	\[\int_{\Gamma_\alpha}\log w(x)\,\omega(u_0,dx,\Omega_\alpha) = \int_{\Gamma_\alpha}\log w_0(x)\,\omega(u_0,dx,\Omega_\alpha) +\int_{\Gamma_\alpha}\log |x-u_j|\,\omega(u_0,dx,\Omega_\alpha).\]
    By \cite[Sect.II.4]{ST97}, we see that
	\begin{equation}
		\exp\left(\int_{\Gamma_\alpha}\log |x-u_j|\,\omega(u_0,dx,\Omega_{\alpha})\right) = 
		|u_0-u_j| e^{-g_{\Omega_\alpha}(u_0,\infty)},
		\label{eq:frostman_theorem}
	\end{equation}
	which simplifies to $\Cap(\Gamma_\alpha)$ when $u_0=\infty$. This implies that for any $u_0\in \Omega_\alpha$, we have
	\begin{equation}
		\liminf_{n\rightarrow \infty} 
		\|w T_n^{w}(\cdot,u_0)\|_{\Gamma_{\alpha}}e^{ng_{\Omega_\alpha}(u_0,\infty)}
		\geq k_{\Omega_\alpha}(u_0,u_0)^{-1}
		\exp\left(\int_{\Gamma_\alpha} \log w(x)\,\omega(u_0,dx,\Omega_{\alpha})\right).
		\label{eq:liminf_weight_zero}
	\end{equation}
To estimate the corresponding $\limsup$, we define $w_\epsilon$ as
	\begin{equation}
	w_\epsilon(u) = w_0(u)(|u-u_j|+\epsilon),
	\end{equation}
	and observe that $w_\epsilon(u)>w(u)$ for any $u\in \Gamma_\alpha$. Additionally, $w_\epsilon$ satisfies the assumptions necessary to ensure that 
	\eqref{eq:limit_chebyshev_norm} holds. Consequently, we have
	\[
		\|w T_n^{w}(\cdot,u_0)\|_{\Gamma_\alpha}\leq \|w_{\epsilon} T_n^{w_{\epsilon}}(\cdot,u_0)\|_{\Gamma_\alpha}\sim e^{-ng_{\Omega_\alpha}(u_0,\infty)}k_{\Omega_\alpha}(u_0,u_0)^{-1}\exp\left(\int_{\Gamma_\alpha} \log w_{\epsilon}(x)\,\omega(u_0,dx,\Omega_{\alpha})\right).
	\]
	As $\epsilon\rightarrow 0$, the expression
	\[\int_{\Gamma_\alpha}\log w_\epsilon(x)\,\omega(u_0,dx,\Omega_{\alpha}) = \int_{\Gamma_\alpha}\log w_0(x)\,\omega(u_0,dx,\Omega_{\alpha})+\int_{\Gamma_\alpha}\log(|u-u_j|+\epsilon)\,\omega(u_0,dx,\Omega_{\alpha})\]
	converges, by the monotone convergence theorem, to
	\[\int_{\Gamma_\alpha}\log w(x)\,\omega(u_0,dx,\Omega_{\alpha}).\]
	We thus find that 
	\begin{equation}
	\limsup_{n\rightarrow \infty}
	   \|w T_n^{w}(\cdot,u_0)\|_{\Gamma_{\alpha}}e^{ng_{\Omega_\alpha}(u_0,\infty)}	   
	   \leq k_{\Omega_\alpha}(u_0,u_0)^{-1}
	   \exp\left(\int_{\Gamma_\alpha} \log w(x)\,\omega(u_0,dx,\Omega_{\alpha})\right).
	   \end{equation}
	Together with \eqref{eq:liminf_weight_zero}, this establishes that
	\begin{equation} \label{As}		
	\|w T_n^{w}(\cdot,u_0)\|_{\Gamma_\alpha}\sim  e^{-ng_{\Omega_\alpha}(u_0,\infty)}k_{\Omega_\alpha}(u_0,u_0)^{-1}\exp\left(\int_{\Gamma_\alpha} \log w(x)\,\omega(u_0,dx,\Omega_{\alpha})\right).
	\end{equation}

The use of induction extends \eqref{As} to any weight of the form 
	\begin{equation} w(u) = w_0(u)|u-u_j|^m,\end{equation}
	where $m$ is a positive integer. Similarly, if $m$ is a positive integer and the weight has the form
	\[w(u) = w_0(u)|u-u_j|^{-m},\]
	it follows that
	\[T_{n+m}^w(u,u_0) =  \begin{cases}
	(u-u_j)^{m}(u_0-u_j)^{-m}T_{n}^{w_0}(u,u_0),& u_0\neq \infty, \vspace{0.15cm} \\ 
	\qquad \;(u-u_j)^{m}T_{n}^{w_0}(u,\infty), & u_0 = \infty. 
  \end{cases}
\]
	Hence, a brief computation confirms the validity of \eqref{As} in this case as well. 
	
	To complete the proof, it remains to verify that \eqref{As} holds for weights of the form 
	\begin{equation}w(u) = w_0(u)|u-u_j|^s,\end{equation}
	where $0<s<1$. To this end, we introduce the weight functions, parameterized by $\delta>0$:
	\[w_l(u) = \begin{cases}
		w_0(u)|u-u_j|, & |u-u_j|<\delta,\\
		\quad \; \; \, w(u), & |u-u_j|\geq \delta, \\
	\end{cases}\]
	and
	\[w_u(u) = \begin{cases}
		w_0(u), & |u-u_j|<\delta,\\
		\, w(u), & |u-u_j|\geq \delta. \\
	\end{cases}\]
	It is a simple matter to verify that $w_l(u)\leq w(u)\leq w_u(u)$ for all $u\in \Gamma_\alpha$ and, consequently, 
	\[\|w_l T_n^{w_l}(\cdot,u_0)\|_{\Gamma_\alpha}\leq\|w T_n^{w}(\cdot,u_0)\|_{\Gamma_\alpha}\leq\|w_u T_n^{w_u}(\cdot,u_0)\|_{\Gamma_\alpha}.\]
	From this, it follows that
	\begin{multline*}
		\quad k_{\Omega_\alpha}(u_0,u_0)^{-1}
		\exp\left(\int_{\Gamma_\alpha}\log w_l(x)\,\omega(u_0,dx,\Omega_{\alpha})\right)
		\leq \liminf_{n\rightarrow\infty}
		\|w T_n^{w}(\cdot,u_0)\|_{\Gamma_{\alpha}}e^{ng_{\Omega_\alpha}(u_0,\infty)} \\
	        \leq\limsup_{n\rightarrow\infty}
	        \|w T_n^{w}(\cdot,u_0)\|_{\Gamma_{\alpha}}e^{ng_{\Omega_\alpha}(u_0,\infty)}
	        \leq k_{\Omega_\alpha}(u_0,u_0)^{-1}
	        \exp\left(\int_{\Gamma_\alpha}\log w_u(x)\,\omega(u_0,dx,\Omega_{\alpha})\right). \quad
	\end{multline*}
	Since
	\[
	\int_{\Gamma_\alpha}\log w_l(x)\,\omega(u_0,dx,\Omega_{\alpha})
	\leq \int_{\Gamma_\alpha}\log w(x)\,\omega(u_0,dx,\Omega_{\alpha})
	\leq \int_{\Gamma_\alpha}\log w_u(x)\,\omega(u_0,dx,\Omega_{\alpha}),	
	\]
	the monotone convergence theorem implies that as $\delta\rightarrow 0$, the upper and lower bounds 
	converge to the same limit. This completes the proof.
\end{proof}

Theorem \ref{thm:main_norm_asymptotics} provides an asymptotic formula for the weighted norm of a broad class of weighted Chebyshev polynomials on circular arcs. To our knowledge, this is the first example of such a result for Jordan arcs beyond the classical case of an interval.

Interestingly, the key argument that allows us to extend the result from positive continuous weights to weights with zeros is not specific to intervals or circular arcs; rather, it applies to any 
Jordan arc and, in fact, to more general compact sets. More precisely, let $K$ be a regular (for potential theory) compact set for which there exists a constant $C_K$ such that
	\begin{equation} \label{W0}
	\lim_{n\rightarrow \infty}\cW_n(K,w_0)= 
	C_K\exp\left(\int_{K} \log w_0(x)\,d\mu_{K}(x)\right)
	\end{equation}
	for every continuous positive weight $w_0:K\to (0,\infty)$. Then the method used in the proof of Theorem \ref{thm:main_norm_asymptotics} can be adapted to show that \eqref{W0} remains valid for more general weights of the form
	\begin{equation}
	   w(x) = w_0(x)\prod_{j=1}^{m}|x-x_j|^{s_j},
	\end{equation}   
	where $s_j\in \R$ and $x_j\in K$ for $j=1,\dots,m$.

\section*{Acknowledgement}
These results were established during a SQuaRE meeting at the American Institute of Mathematics (AIM) in Pasadena, CA. We extend our heartfelt gratitude to AIM for their generous hospitality and for creating such an inspiring scientific environment.

\bibliographystyle{plain} 

\end{document}